%% file: main.tex
\documentclass[10pt]{article}

\usepackage{amsmath,amssymb}
\usepackage{amsthm,amsfonts}
\usepackage{graphicx}
\usepackage{comment}
\usepackage{tikz}
\usepackage{pgfplots}
\usepackage{authblk}
\usepackage{float}% Antes nofloat importante
\usepackage{subcaption}
\usepackage[left=2.8cm,right=2.8cm, top=3.1cm, bottom=2.8cm]{geometry}
\usepackage[%
  colorlinks=true,%
  linkcolor=blue,%
  citecolor=blue,%
  urlcolor=blue%
]{hyperref}
\usepackage{hypcap}%importante

\usetikzlibrary{arrows.meta,calc,decorations.pathreplacing}
\pgfplotsset{compat=1.18}
\DeclareCaptionLabelFormat{subfiglabel}{#2)}
\newcommand{\mphrase}[1]{\vcenter{\hbox{\shortstack[c]{#1}}}}
\newcommand{\ControlC}{\textsf{\textup{\textbf{c}}}}
\newcommand{\ControlS}{\textsf{\textup{\textbf{s}}}}

\newtheorem{theo}{Theorem}[section]

\newtheorem{coro}[theo]{Corollary}
\newtheorem{prop}[theo]{Proposition}

\newtheorem{remark}[theo]{Remark}
\newcommand{\R}{\hbox{\rm I \kern -5pt R}}     % simbolo de Reales
\newcommand{\p} {\hbox{\rm I \kern -5pt P}}
\def\n        {\nabla}

\numberwithin{equation}{section}

\title{Modeling of an ODE-constrained optimization problem describing tumor dynamics, and numerical approximation via sequential
physics-informed neural networks }

\author[a]{Juan J. Forero-Herna\'ndez\thanks{E-mail: \href{mailto:juan2258036@correo.uis.edu.co}{ juan2258036@correo.uis.edu.co}}}
\author[a]{\'Elder J. Villamizar-Roa\thanks{E-mail (Corresponding author) \href{mailto:jvillami@uis.edu.co}{jvillami@uis.edu.co}}}

\affil[a]{Universidad Industrial de Santander, Escuela de Matem\'{a}ticas,  Carrera 29 Calle 9, A.A. 678, CP 680002, Bucaramanga, Colombia.}

\date{}

\IfFileExists{algorithm.sty}{\usepackage{algorithm}}{%
  \newcounter{algorithm}
  \renewcommand{\thealgorithm}{\arabic{algorithm}}
}
\IfFileExists{algpseudocode.sty}{\usepackage{algpseudocode}}{%
  \newenvironment{algorithmic}[1][]{%
    \begin{list}{}{\setlength{\leftmargin}{1.2em}}%
    \item[]%
  }{%
    \end{list}%
  }%
  \newcommand{\State}{\item[]}%
}

\newlength{\algindentone}
\newlength{\algindenttwo}
\newcommand{\algline}[1]{%
    \State \parbox[t]{\dimexpr\linewidth\relax}{#1}%
}

\newcommand{\alglinei}[1]{%
    \State \hspace*{\algindentone}%
    \parbox[t]{\dimexpr\linewidth-\algindentone\relax}{#1}%
}

\newcommand{\alglineii}[1]{%
    \State \hspace*{\algindenttwo}%
    \parbox[t]{\dimexpr\linewidth-\algindenttwo\relax}{#1}%
}

\makeatletter
\providecommand{\fname@algorithm}{Algorithm}
\newenvironment{breakablealgorithm}
  {
   \par\noindent
   \refstepcounter{algorithm}
   \hrule height.8pt depth0pt \kern2pt
   \renewcommand{\caption}[2][\relax]{
     {\raggedright\textbf{\fname@algorithm \thealgorithm}\ ##2\par}
     \ifx\relax##1\relax
       \addcontentsline{loa}{algorithm}{\protect\numberline{\thealgorithm}##2}
     \else
       \addcontentsline{loa}{algorithm}{\protect\numberline{\thealgorithm}##1}
     \fi
     \kern2pt\hrule\kern2pt
   }
  }
  {
   \kern2pt\hrule
   \par
  }
\makeatother

\begin{document}
\maketitle
 
%%%%%%%%%%%%%%%%%%%%%%%%%%%%%%%%%%%%%%%%%%%%%%%%%%%%%%%%%%%%%
%%%%%%%%%%%%%%%%%%%%%%%%%%%%%%%%%%%%%%%%%%%%%%%%%%%%%%%%%%%%%
% RESUMEN
%%%%%%%%%%%%%%%%%%%%%%%%%%%%%%%%%%%%%%%%%%%%%%%%%%%%%%%%%%%%%
%%%%%%%%%%%%%%%%%%%%%%%%%%%%%%%%%%%%%%%%%%%%%%%%%%%%%%%%%%%%%

\begin{abstract}
In this paper, we study an optimal control problem related to an ODE model of glioblastoma growth influenced by the oxygen. The  model considers a couple of controls describing the chemotherapy and antiangiogenic therapies. The cost functional aims to reduce the tumor growth, bring the oxygen concentration close to a desired value, and penalize the use of therapies. We solve the optimal control problem, proving the existence of a global optimal control and deriving first-order necessary optimality conditions through the Pontryagin Minimum Principle. For the numerical approximation, we use Physics-Informed Neural Networks (PINNs) to solve the state and adjoint systems, together with a gradient descent method with Armijo line search for the controls. To address the strategy of PINNs we consider the methodology proposed in \cite{Roy}, making a decomposition of the time domain into several subintervals, using different neural networks in each subinterval and enforcing continuity conditions between successive time subintervals. This strategy, called sequential PINN formulations in time, including soft and hard-constrained versions, is compared with the corresponding approximation results of classical solvers and traditional PINNs counterparts.

\vspace{0.3cm}

\noindent{\bf Keywords.} {Tumor Growth, Glioblastoma, Optimal Control, Pontryagin Minimum Principle, Physics-Informed Neural Networks, Sequential PINNs, Armijo Line Search.}\vspace{0.3cm}

%\noindent{\bf AMS subject classifications.} 35K55; 35Q35; 35Q92; 92C17
\noindent{\bf AMS subject classifications.} {34C60; 34H05; 49J15; 49K15; 93C10; 68T07.}
\end{abstract}

%%%%%%%%%%%%%%%%%%%%%%%%s%%%%%%%%%%%%%%%%%%%%%%%%%%%%%%%%%%%%%
%%%%%%%%%%%%%%%%%%%%%%%%%%%%%%%%%%%%%%%%%%%%%%%%%%%%%%%%%%%%%
% CAPÍTULO 1: INTRODUCCIÓN
%%%%%%%%%%%%%%%%%%%%%%%%%%%%%%%%%%%%%%%%%%%%%%%%%%%%%%%%%%%%%
%%%%%%%%%%%%%%%%%%%%%%%%%%%%%%%%%%%%%%%%%%%%%%%%%%%%%%%%%%%%%
\section{Introduction}
The mathematical modeling of cancer represents an important research topic due to the complexity of the disease development, resulting from multiple biological, genetic, and environmental factors. From a medical perspective, the disease arises from genetic mutations that alter normal cell function and induce uncontrolled proliferation, giving rise to what are known as cancerous tumors. Chromosomal abnormalities accumulate, leading to profound genetic instability. In later stages, cancer cells metastasize, allowing them to spread to distant tissues and organs
\cite{Bunz}. A broad class of cancerous tumors corresponds to glioblastomas, which are brain and spinal cord tumors characterized by being highly invasive and having irregular morphologies that cannot be identified with sufficient precision using medical imaging techniques, making it difficult to achieve a sufficiently precise resection \cite{leo1}.\\

In general, to develop effective therapeutic strategies, it is necessary to understand in depth the dynamics of tumor cell density and its interaction with the surrounding microenvironment. The analysis of tumor dynamics under different biological and physical scenarios allows the evaluation of cellular responses to therapeutic interventions \cite{Anderson2005, Natesan}. From the mathematical point of view, the tumor dynamics is often expressed through systems of ordinary differential equations and partial differential equations. In particular, in the context of dynamics of the glioblastoma, some important elements stand out. First, the selection of variables governing cellular dynamics and the manner in which reactions are represented, the incorporation or omission of spatial diffusion processes, and the accurate formulation of therapeutic mechanisms aimed at constraining tumor growth. In particular, the complex dynamics of tumor cell infiltration into brain tissue poses significant challenges
from both clinical and scientific perspectives, for which mathematical modeling has  become an indispensable tool to elucidating the mechanisms of tumor propagation and for exploring potential therapeutic strategies 
\cite{forero-hernandez}. In relation with the variables governing cellular dynamics, some studies have remarked the role of the oxygen in the cellular proliferation, as well as in the aggressiveness and invasiveness of glioblastoma (see, e.g. \cite{Gomez2017, leo1, Swanson}). Therefore, inspired by \cite{forero-hernandez}, in this paper, we propose and analyze an optimal control problem
to deal with the treatment of the tumor invasion, by using chemotherapic and antiangiogenic controls
where the state equations describe the evolution of glial cells in the presence of oxygen. Precisely, we consider the following ODE state-control system

\begin{equation}
	\begin{cases}
		u' =  \overbrace{\rho(\sigma)u(\alpha-u)}^{\mbox{\tiny{Logistic Growth}}} - \overbrace{\kappa \ControlC u \sigma}^{\mbox{\tiny{Chemotherapy}}},\\
		\sigma'= \overbrace{P_{er}S_{v}(\beta-\sigma)}^{\mbox{\tiny{Reaction-Production}}}-\overbrace{{\frac{A_{ox}u\sigma}{k_{ox}+\sigma}}}^{\mbox{\tiny{Consumption}}}+\overbrace{(1-\ControlS) S_c u}^{\mbox{\tiny{\shortstack{Antiangiogenic\\Therapy}}}
}, \\
        u(0)= u_0, \sigma(0)= \sigma_0,
	\end{cases}
\label{eq:SistemaControlado}
\end{equation}
where $u:[0,T] \to \mathbb{R}$ and $\sigma :[0,T] \to \mathbb{R}$ are the unknowns representing the tumor cell density and the oxygen concentration, respectively. Here $u_0\geq0$ and $\sigma_0\geq 0$ denote the initial tumor cells density and initial oxygen concentration, respectively. The first equation of \eqref{eq:SistemaControlado} describes the evolution of the tumor cells. The term $\rho(\sigma)u(\alpha - u)$ represents a logistic growth modulated by the oxygen concentration. The constant $\alpha > 0$ represents the carrying capacity, while $\rho(\sigma)$ is the oxygen-dependent proliferation rate. The term $-\kappa \ControlC u \sigma$ models the action of the cytotoxic therapy under the tumor cells density. More exactly, $\ControlC$ is a time-dependent function with $0\leq \ControlC\leq 1$ representing the administration of a cytotoxic drug
(chemotherapy) which reduces the tumor cell population. The parameter $\kappa > 0$ measures the efficiency of the control $\ControlC$, and the factor $-u\sigma$ reflects that the decrease of tumor density depends on both the available tumor population and the oxygen level. The second equation in \eqref{eq:SistemaControlado} describes the evolution of oxygen. The term $P_{er}S_v(\beta -\sigma)$ represents the oxygen supply from the vasculature toward the reference level $\beta > 0$. The constant $P_{er} > 0$ denotes the vascular permeability and $S_v > 0$ the vascular density. The term $-\frac{A_{ox}u\sigma}{k_{ox}+\sigma}$ describes oxygen consumption by tumor cells according to Michaelis-Menten kinetics, where $A_{ox}$ and $k_{ox}$ are positive constants. Finally, the source term $(1-\ControlS)S_cu$ models the contribution of angiogenic factors released by the tumor. The constant $S_c > 0$ describes the strength of the time-dependent control $0\leq \ControlS\leq 1$ (antiangiogenic therapy).\\

The proliferation rate is described by a function
\(\rho:\mathbb{R}^+\to\mathbb{R}\) such that \(\rho\in C^1([0,\infty))\) and
\(\rho'\) is bounded. We also assume that \(\rho\) remains strictly positive, that is, 
$0<\rho_{\min}\leq \rho(s),$ $s\geq 0,$ for some constant $\rho_{\min}.$ A type of proliferation function used in the literature is given by (see, e.g. \cite{Gomez2017, leo1}):
\[
\rho(\sigma)=\frac{\hat{\rho}}{\alpha}
\left(
\frac{\sigma}{\beta}
+b\left(1-\frac{\sigma}{\beta}\right)
\right).
\]
Here, \(\hat{\rho}\) denotes the maximal proliferation rate in normoxic
conditions, while the parameter \(b\in(0,1]\) modulates the degree to which low oxygen levels reduce proliferation.\\

From a biomedical standpoint, the primary objective is to suppress or inhibit tumor progression through multimodal therapeutic strategies, including chemotherapy and angiogenesis inhibition. Nevertheless, these interventions are frequently associated with substantial adverse effects, such as fatigue, gastrointestinal disturbances, hematopoietic suppression, among others, which can result in immunosuppression \cite{Joshi2009}. Consequently, it is appropriate to formulate an optimal control framework designed to minimize tumor proliferation while concurrently modulating oxygen concentration toward a prescribed physiological target. An additional objective within this framework is the minimization of pharmacological dosages to reduce the severity of treatment-related toxicities. Therefore, the aim of this paper 
propose and analyze an optimal control problem associated with the nonlinear system (\ref{eq:SistemaControlado}), modeling the glioblastoma growth influenced by the presence of oxygen where the
controls are two different (chemotherapy and antiangiogenic) therapies.\\

We analyze the
existence and boundedness of solution of the state system and subsequently, we propose and
analyze a related optimal control problem. We consider a cost functional that reduces
the tumor growth, brings the oxygen concentration close to a desired value, and accounts for the therapeutic cost. We prove the existence of a global
optimal solution and derive necessary first-order optimality conditions via the Pontryagin Minimum Principle. Finally, we
propose a methodology for approximating the optimal control based on the physics-informed neural networks (PINNs) as a novel approach to solve the ODE-constrained control problem. Explicitly, we use the PINN methodology to approximate the state and the adjoint systems, combined with a minimization procedure of the cost functional making use of the gradient-descent algorithm with Armijo step-size. Within the architecture of PINNs, we consider the sequential methodology proposed by \cite{Roy}, through which the time domain is decomposed into several subdomains, using different neural networks in each subdomain. In order to ensure the accuracy in the prediction of the time-dependent problem, it is necessary to impose time compatibility conditions between successive time subintervals through soft- and hard-constrained formulations, using a loss term and an ansatz solution, respectively. This approach has been used to approximate some time-dependent problems including advection, Allen–Cahn, and Korteweg–de Vries equations. In this paper, we adapt the sequential PINNs strategy to provide a numerical approximation of the ODE-constrained optimization problem \eqref{eq:ProblemaDeControl}, in which, not only the approximation of the state equations is required, but also of the corresponding adjoint system. This is a point to emphasize, given that the accuracy of the adjoint system depends on a good approximation of the state system in the whole interval, in particular, on the final time. The numerical experiments
conducted with the proposed method show that
the architecture of the sequential neural network, provide accurate predictions that take advantage of the flexibility of deep learning techniques, compared with the corresponding approximation results of classical solvers and traditional PINNs counterparts. 
\\ 

The content of this paper is organized as follows. In Section 2, we establish and analyze the optimal control problem. We prove the existence and boundedness of solutions for the state system \eqref{eq:SistemaControlado}, and the existence of an optimal control, as well as we derive first-order necessary optimality conditions through the Pontryagin Minimum Principle. In Section 3, we present the PINN framework used in this work, including fully connected neural networks, standard PINNs, sequential PINNs in time, and the soft and hard-constrained formulations. We also give an error estimate for the sequential PINN approximation. In Section 4, we present the numerical algorithm for the state-adjoint PINN approximations, and  we describe the computational setup. Additionally, we conduct numerical experiments by solving the optimal control problem \eqref{eq:ProblemaDeControl} in several scenarios, analyzing the error and the training duration differences between Hard-Constrained sequential PINNs and Soft-Constrained Sequential PINNs, and the corresponding approximation via a classical solver. Finally, in Section 5, we summarize the conclusions of this work.

\section{Optimal Control Problem}
In this section, we formulate and solve an optimal control problem associated with the nonlinear ODE system (\ref{eq:SistemaControlado}). The objective is to control the tumor growth under the action of therapies. We consider an optimization problem where the cost functional aims to reduce the tumor growth, bring the oxygen concentration close to a desired value, and at the same time use the lowest possible doses of therapeutic load to reduce adverse effects. 
\subsection{Existence and boundedness of solution}
Before establishing the mathematical formulation of the optimal control problem, we get the following result of existence and boundedness for the solution of the state-system \eqref{eq:SistemaControlado}.
\begin{theo}\label{teo:ExistenciaSolClasica}
    Let $[\ControlC, \ControlS] \in L^\infty(0,T)^2$ with $0\leq \ControlC(t), \ControlS(t) \leq 1$ a.e. $t \in [0,T]$. Assume that $0 \leq u_0 \leq \alpha$ and $0 \leq \sigma_0 \leq C_\sigma$, where $C_\sigma := \beta + \frac{S_c\alpha}{P_{er}S_v}$. Then the state-system \eqref{eq:SistemaControlado} admits a unique absolutely continuous solution $[u,\sigma]$ on $[0,T]$. Moreover, the solution $[u,\sigma]$ satisfies the bounds
    \[
    0\leq u(t) \leq \alpha \qquad and \qquad 0 \leq \sigma(t) \leq C_\sigma, \qquad \text{a.e. } t\in[0,T].
    \]
\end{theo}
\begin{proof}
Let $[\ControlC,\ControlS]\in L^\infty(0,T)^2$ be fixed, with $0\leq \ControlC(t),\ControlS(t)\leq 1$ a.e. $t\in[0,T]$. We first observe that the right-hand side of system \eqref{eq:SistemaControlado} is measurable in $t$ and locally Lipschitz with respect the state variables $[u,\sigma]$. Indeed, the controls $\ControlC$ and $\ControlS$ are measurable and bounded, the function $\rho$ is of class $C^1$ and therefore locally Lipschitz, and the term $\frac{\sigma}{k_{ox}+\sigma}$ is locally Lipschitz for $\sigma>-k_{ox}$. Hence, by the Carathéodory existence and uniqueness theory \cite[Chapter 2]{caratheodory}, there exists a unique maximal solution $[u,\sigma]$, absolutely continuous on $[0,T_{\max})$ for some maximal time $T_{\max}>0$ and satisfying \eqref{eq:SistemaControlado} at almost every $t$. We shall show that this solution remains bounded on $[0,T_{\max})$ and therefore, it can be extended to the interval $[0,T]$.\\

We can use a comparison argument to get the non-negativity and the upper bound of the solution. Indeed, observe that if $v(t)=0$, it satisfies $v'(t)=0=f_u(t,0,\sigma),$ and $v(0)=0\leq u_0.$ where $f_u(t,u,\sigma) = \rho(\sigma)u(\alpha - u) - \kappa \ControlC u \sigma$. Thus, $v(t)$ is a subsolution of \eqref{eq:SistemaControlado}$_1$ and therefore, by a comparison argument, we obtain $0\leq u(t)$ a.e. $t\in [0,T].$ Consider $f_\sigma(t,u,\sigma)=P_{er}S_v(\beta-\sigma(t))-\frac{A_{ox}u(t)\sigma(t)}{k_{ox}+\sigma(t)}+(1-\ControlS(t))S_cu(t).$ Then, \eqref{eq:SistemaControlado}$_2$ is expressed as $\sigma'(t)=f_\sigma(t,u(t),\sigma(t)).$ 
Using that $u$ is nonnegative and $0\leq \ControlS \leq 1$,  the constant function $v(t)=0$ satisfies $v'(t)=0\leq f_\sigma(t,u(t),0),\ v(0)=0\leq \sigma_0.$ Again, applying a comparison argument, we conclude that $0\leq \sigma(t)$ a.e. $t\in [0,T].$ Next, we prove the upper bound for $u$. Since $u\geq 0$, $\sigma\geq 0$ and $\ControlC\geq 0$, from \eqref{eq:SistemaControlado}$_1$ we have
\[
u'(t)=\rho(\sigma(t))u(t)(\alpha-u(t))-\kappa \ControlC(t)u(t)\sigma(t)
\leq \rho(\sigma(t))u(t)(\alpha-u(t)).
\]
The constant function $\alpha$ solves the scalar equation
\[
\begin{cases}
    y'(t)=\rho(\sigma(t))y(t)(\alpha-y(t)),\\
    y(0)=\alpha.
\end{cases}
\]
Since $u_0\leq \alpha$, by comparison  we get that $u(t)\leq \alpha$ a.e. $t\in [0,T].$  Consequently, $0\leq u(t)\leq \alpha.$ Finally, we derive the upper bound for $\sigma$. Using $u\leq \alpha$, $0\leq \ControlS\leq 1$, and the non-negativity of $u$ and $\sigma$, from \eqref{eq:SistemaControlado}$_2$ we obtain
\[
\sigma'(t)
= P_{er}S_v(\beta-\sigma(t))
-\frac{A_{ox}u(t)\sigma(t)}{k_{ox}+\sigma(t)}
+(1-\ControlS (t))S_cu(t)
\leq P_{er}S_v(\beta-\sigma(t))+S_c\alpha.
\]
Let $C_\sigma:=\beta+\frac{S_c\alpha}{P_{er}S_v}$ and let $z(t)$ be the solution of
\[
\begin{cases}
    z'(t)=P_{er}S_v(C_\sigma-z(t)),\\
    z(0)=C_\sigma.
\end{cases}
\]
Since $\sigma_0\leq C_\sigma$ and $\sigma$ is a subsolution of the previous equation, by a comparison argument, we get $\sigma(t)\leq z(t)=C_\sigma$. Therefore $0 \leq \sigma \leq C_\sigma$ a.e. $t \in [0,T]$.

\end{proof}

In order to establish the optimal control problem, we introduce the following cost functional
\begin{equation}\label{funcionalObjetivo}
    \begin{split}
        J([u(\ControlC,\ControlS), \sigma(\ControlC,\ControlS), \ControlC,\ControlS]) 
        &= \int_0^T \Big[\frac{k_1}{2}u^2 + \frac{k_2}{2} \left(\sigma - \sigma_d \right)^2 + k_3\ControlC + k_4\ControlS\Big](t)dt \\ 
        & \ \ \ + \frac{l_1}{2}u(T)^2 + \frac{l_2}{2}\left(\sigma(T)-\sigma_{dT}\right)^2,
    \end{split}
\end{equation}
where $k_i(t) \in L^\infty(0,T)$ $(i=1,2,3,4)$ are positive and bounded weighting functions, and $l_j \in \mathbb{R}$ $(j=1,2)$ are positive real values. In addition, $\sigma_d(t) \in L^\infty(0,T)$ and $\sigma_{dT} \in \mathbb{R}$ represent the desired state for oxygen levels in $(0,T)$ and at the final time $T$, respectively. Additionally, we define the set of admissible controls $[\ControlC, \ControlS]$, as
\[
\mathcal{U}_{ad} = \Big\{[\ControlC, \ControlS] \in L^\infty(0,T)^2 : 0\leq \ControlC, \ControlS \leq 1, \int_0^T \ControlC(t) dt \leq c_{\max}\Big\},
\]
where $c_{\max}>0$ denotes the maximum total amount of therapy that can realistically be administered with standard drugs. The condition $c_{\max}<T$ means that the treatment cannot be applied at its maximum level for the entire time interval, which is a condition that responds to the reality of the problem. Since the states depend of the controls, we consider the reduced functional by
\[
\widehat{J}([\ControlC,\ControlS]) := J([u(\ControlC,\ControlS),\sigma(\ControlC,\ControlS),\ControlC,\ControlS]).
\]
Therefore, the optimization problem we want to analyze is defined as follows:
\begin{equation}\label{eq:ProblemaDeControl}
    \textnormal{(P)}\quad
\left\{
\begin{array}{ll}
\displaystyle 
\min
& \widehat J([\ControlC,\ControlS])
:=J([u(\ControlC,\ControlS),\sigma(\ControlC,\ControlS),\ControlC,\ControlS]) \\[1.5ex]
\textnormal{s.t.}
&
\left\{
\begin{array}{l}
[u(\ControlC,\ControlS),\sigma(\ControlC, \ControlS)] \textnormal{ is the solution of the state system } \eqref{eq:SistemaControlado} \\[0.4ex]
\textnormal{associated with the admissible controls } [\ControlC,\ControlS]\in\mathcal U_{\mathrm{ad}}.
\end{array}
\right.
\end{array}
\right.
\end{equation}
If the minimum is just in a neighborhood, then it is a {\it local optimal control.}
\subsection{Existence of optimal control}
We present the following result regarding the existence of at least one optimal control for the optimal control problem \eqref{eq:ProblemaDeControl}:
\begin{theo}
    Under the hypotheses of Theorem \ref{teo:ExistenciaSolClasica}, there is at least a global optimal solution $[\ControlC^*, \ControlS^*] \in \mathcal{U}_{ad}$ for the controlled system \eqref{eq:SistemaControlado}, that is,
    \[
    \widehat J([\ControlC^*, \ControlS^*]) = \min_{[\ControlC,\ControlS]\in\mathcal{U}_{ad}}\widehat J([\ControlC,\ControlS]).
    \]
\end{theo}
\begin{proof}
    The result is obtained by applying Theorem 23.11 in \cite{Clarke}. For that, we verify assumptions (a)--(f) for the functional \(\widehat J\) according the statement of Theorem 23.11 in \cite{Clarke}. Since the solutions of the system \eqref{eq:SistemaControlado} are bounded, then the condition (a) is satisfied. Moreover, the set of admissible controls $\mathcal{U}_{ad}$ is closed and convex, and the terminal cost is continuous with respect to the state variables $[u,\sigma]$. Therefore, the conditions (b), (c), and (f) are satisfied. Since the cost function is convex, continuous and since $u$ and $\sigma$ are bounded then the condition (d) is satisfied. Condition (e) follows directly from the given initial conditions. Moreover, any $[\ControlC, \ControlS] \in \mathcal{U}_{ad}$ yields an admissible process with finite $\widehat J$, completing the hypothesis of Theorem 23.11.
\end{proof}

\subsection{First-order necessary optimality conditions}

Now, we derive the first-order necessary optimality conditions associated with the optimal control problem. For this, we apply the Pontryagin Minimum Principle, which provides a description of the
optimal controls in terms of the Hamiltonian function and the associated
adjoint system.\\

The Hamiltonian associated with the optimal control problem \eqref{eq:ProblemaDeControl} is defined by
\begin{equation}\label{eq:Hamiltonian_PMP}
\begin{aligned}
\mathcal H(u,\sigma,\ControlC,\ControlS,p_1,p_2,\lambda_0)
&=
p_1\Big(\rho(\sigma)u(\alpha-u)-\kappa \ControlC u\sigma\Big)\\
&\quad
+p_2\left(
P_{er}S_v(\beta-\sigma)
-\frac{A_{ox}u\sigma}{k_{ox}+\sigma}
+(1-\ControlS)S_cu
\right)\\
&\quad
+\lambda_0
\left[
\frac{k_1}{2}u^2
+
\frac{k_2}{2}(\sigma-\sigma_d)^2
+
k_3\ControlC+k_4\ControlS
\right],
\end{aligned}
\end{equation}
where $p_1$ and $p_2$ are the adjoint variables associated with the state variables $u$ and $\sigma$, respectively, and $\lambda_0 \geq 0$ is the multiplier related to the cost functional. By the Pontryagin Minimum Principle, there exist $\lambda_0 \geq 0$ and an absolutely continuous adjoint variable $p=(p_1,p_2) : [0,T] \to \mathbb{R}^2$ not identically zero, such that the following conditions hold:
\begin{enumerate}
\item The adjoint variables satisfy the adjoint system. This system is obtained from
\[
-p_1'(t)=\partial_u\mathcal H(u^*,\sigma^*,\ControlC^*,\ControlS^*,p_1,p_2)(t),
\qquad
-p_2'(t)=\partial_\sigma\mathcal H(u^*,\sigma^*,\ControlC^*,\ControlS^*,p_1,p_2)(t),
\]
together with the terminal conditions induced by the terminal cost. Consequently, we obtain the adjoint system:
\begin{equation}\label{eq:adjoint_system_ode}
\left\{
\begin{aligned}
-p'_1
&=
\lambda_0 k_1u^*
+
\left(
\rho(\sigma^*)(\alpha-2u^*)
-\kappa \ControlC^*\sigma^*
\right)p_1 + \left(
-\frac{A_{ox}\sigma^*}{k_{ox}+\sigma^*}
+(1-\ControlS^*)S_c
\right)p_2,
\\[1mm]
-p'_2
&=
\lambda_0 k_2\big(\sigma^*-\sigma_d\big)
+
\left(
\rho'(\sigma^*)u^*(\alpha-u^*)
-\kappa \ControlC^*u^*
\right)p_1 
+
\left(
-P_{er}S_v
-\frac{A_{ox}u^*k_{ox}}{(k_{ox}+\sigma^*)^2}
\right)p_2,
\\[1mm]
p_1(T)& =\lambda_0 l_1u^*(T), \qquad
p_2(T)=\lambda_0 l_2\left(\sigma^*(T)-\sigma_{dT}\right).
\end{aligned}
\right.
\end{equation}
\item The nontriviality condition holds:
\begin{equation}\label{eq:PontryaginNoTrivialidad}
(\lambda_0, p(t)) \neq (0,0) \quad \text{for every }t\in [0,T].
\end{equation}
Using the uniqueness of the solution to the adjoint system \eqref{eq:adjoint_system_ode}, it is possible to show that the multiplier $\lambda_0$ is nonzero. Indeed, suppose by contradiction that $\lambda_0 = 0$. In that case, the terminal conditions in \eqref{eq:adjoint_system_ode} give $p_1(T) = p_2(T) = 0$, and the adjoint system becomes the linear homogeneous system
\[
\left\{
\begin{aligned}
-p'_1
&=
\left(
\rho(\sigma^*)(\alpha-2u^*)
-\kappa \ControlC^*\sigma^*
\right)p_1 + \left(
-\frac{A_{ox}\sigma^*}{k_{ox}+\sigma^*}
+(1-\ControlS^*)S_c
\right)p_2,
\\[1mm]
-p'_2
&=
\left(
\rho'(\sigma^*)u^*(\alpha-u^*)
-\kappa \ControlC^*u^*
\right)p_1 
+
\left(
-P_{er}S_v
-\frac{A_{ox}u^*k_{ox}}{(k_{ox}+\sigma^*)^2}
\right)p_2,
\\[1mm]
p_1(T)&=p_2(T)=0.
\end{aligned}
\right.
\]
Since $u^*$ and $\sigma^*$ are bounded, $\rho\in C^1([0,\infty))$ with $\rho'$ bounded, $\sigma^*\geq 0$, and $\ControlC^*,\ControlS^*\in L^\infty(0,T)$, the coefficients of this system belong to $L^\infty(0,T)$. Hence the terminal value problem has a unique absolutely continuous solution, which is $p_1\equiv p_2\equiv 0$ on $[0,T]$. Therefore $(\lambda_0,p(t))=(0,0)$, which contradicts the non-triviality condition \eqref{eq:PontryaginNoTrivialidad}. Consequently, $\lambda_0>0$.

\item The Hamiltonian condition is satisfied over the admissible set $\mathcal{U}_{ad}$. Since this set contains the global restriction
\[
\int_0^T\ControlC(t)\,dt\leq c_{\max},
\]
we write this condition in variational form.
\end{enumerate}
Since $\lambda_0>0$, without loss of generality, we take $\lambda_0=1$, so that \eqref{eq:adjoint_system_ode} (with $\lambda_0=1$) is the adjoint system that will be used from now on, both in the optimality condition below and in the numerical approximation of Section 4. Then, the optimality condition is written as

\begin{equation}\label{eq:VI_PMP_expanded}
\int_0^T
\left[
\left(k_3-\kappa u^*\sigma^*p_1\right)
\big(\ControlC-\ControlC^*\big)
+
\left(k_4-S_cu^*p_2\right)
\big(\ControlS-\ControlS^*\big)
\right]dt
\geq 0,
\end{equation}
for every $[\ControlC,\ControlS]\in\mathcal U_{ad}$. Note that $k_3-\kappa u^*\sigma^*p_1$ and  $k_4 - S_c u^*p_2$ correspond to the gradient of $\widehat{J}$ with respect to the control $\ControlC$ and $\ControlS$, respectively, namely,
\begin{equation}\label{eq:gradientes}
\nabla_\ControlC \widehat{J}([\ControlC^*,\ControlS^*]) = k_3(t)-\kappa u^*(t)\sigma^*(t)p_1(t), \qquad \nabla_\ControlS \widehat{J}([\ControlC^*,\ControlS^*]) = k_4(t)-S_c u^*(t)p_2(t).
\end{equation}

The projection onto $\mathcal{U}_{ad}$ is therefore not purely pointwise for the cytotoxic control $\ControlC$, because of the integral restriction. We will use the following projection result.

\begin{prop}\label{prop:proyeccion}(See \cite{Fernandez-RestriccionGlobalControl})
Let us suppose that $v_{\max},y_{\max}\in(0,+\infty)$ and
$\mu\in L^2(0,\tau)$, with $\mu(t)>0$ a.e. $t\in[0,\tau]$, verifying
\begin{equation}\label{eq:rho-condition}
    \int_0^\tau \mu(t)\,dt > \frac{y_{\max}}{v_{\max}}.
\end{equation}

Let us consider the set
\[
\mathbb{K}
=
\left\{
v\in L^2(0,\tau):
0\leq v(t)\leq v_{\max}\ \text{a.e. } t\in[0,\tau],
\quad
\int_0^\tau v(t)\mu(t)\,dt\leq y_{\max}
\right\},
\]
and denote by $\operatorname{Proj}_{\mathbb{K}}$ the projection operator onto
$\mathbb{K}$. For each $h\in L^2(0,\tau)$, we have that:
\begin{enumerate}
    \item If $ \int_0^\tau \operatorname{Proj}_{[0,v_{\max}]}(h(t))\mu(t)\,dt \leq y_{\max},$
    then $\operatorname{Proj}_{\mathbb{K}}(h)=\operatorname{Proj}_{[0,v_{\max}]}(h).$

    \item If $\int_0^\tau \operatorname{Proj}_{[0,v_{\max}}](h(t))\mu(t)\,dt > y_{\max},$ then there exists $\gamma>0$ such that
    \[
        \operatorname{Proj}_{\mathbb{K}}(h) = \operatorname{Proj}_{[0,v_{\max}]}\bigl(h-\gamma \mu\bigr), \qquad \text{and} \qquad \int_0^\tau \operatorname{Proj}_{\mathbb{K}}(h)(t)\mu(t)\,dt = y_{\max}.
    \]
\end{enumerate}
\end{prop}

%%%%%%%%%%%%%%%%%%%%%%%%%%%%%%%%%%%%%%%%%%%%%%%
% CAPÍTULO 3: PINNs
%%%%%%%%%%%%%%%%%%%%%%%%%%%%%%%%%%%%%%%%%%%%%%% 
\section{Preliminaries of PINNs}

In this section we introduce a general mathematical framework for the neural networks that we used in this paper. We consider the following general system of $d$ first-order non-linear ODE:
\begin{equation}\label{eq:EDO_abs}
\begin{cases}
\mathbf{u}'(t) = \mathbf{F}(t,\mathbf{u}(t)), \qquad \forall t \in [t_0, T],\ T \geq t_0,\\
\mathbf{u}(t_0) = \mathbf{u}_0,
\end{cases}
\end{equation}
where $\mathbf{u}_0 \in \mathbb{R}^d$ is the initial condition, $\mathbf{u} : [t_0,T] \to \mathbb{R}^d$ is the unknown state and $\mathbf{F} : [t_0, T] \times \mathbb{R}^d \to \mathbb{R}^d$ is a given the nonlinear function. Therefore, for convenience, we introduce the nonlinear differential operator
\[
\mathcal{R}[\mathbf{u}]:= \mathbf{u}'(t) - \mathbf{F}(t,\mathbf{u}(t)), \quad \forall t \in [t_0, T].
\]
Thus, solving the system \eqref{eq:EDO_abs} is equivalent to finding $\mathbf{u}\in [C^1([t_0, T])]^d$ such that 
\[
\mathcal{R}[\mathbf{u}](t) = \mathbf{0}, \quad \forall t\in [t_0, T], \qquad {\mbox{and}}\ \qquad \mathcal{I}[\mathbf{u}]:= \mathbf{u}(t_0) - \mathbf{u}_0 = \mathbf{0}.
\]

\subsection{Fully Connected Feedforward Neural Network} The objective is to train a neural network that approximates the solution to system \eqref{eq:EDO_abs}; for that, let $D \geq 2$ be the depth of the network, and consider $n_0 = 1,$  $n_D = d$, and 
\[
n_1, \dots, n_{D-1} \in \mathbb{N}
\]
the widths of the hidden layers. The set of trainable parameters is denoted by $\theta := \{(W^l, b^l)\}^D_{l=1}$, where
\[
W^l \in \mathbb{R}^{n_l \times n_{l-1}}, \quad b^l \in \mathbb{R}^{n_l}, \qquad {\mbox{for}} \quad l = 1,\dots, D.
\]
Therefore, given a $t \in [t_0,T]$, the network is defined recursively by
\[
a^{(0)}(t):= t,
\]
and 
\[
z^{(l)}(t) := W^{l}a^{(l-1)}(t)+b^l, \qquad a^{(l)}(t) := \varphi_l (z^{(l)}(t)), \qquad {\mbox{for}} \quad l=1,\dots, D-1,
\]
where $\varphi_l$ denotes the activation function of the $l$-th hidden layer. The output layer is then given by
\[
\mathcal{N}_\theta (t) := W^Da^{(D-1)}(t) + b^D \in \mathbb{R}^d.
\]
Thus, we can define a trial function constructed from $\mathcal{N}_\theta$, denoted by $\mathbf{u}_\theta : [t_0, T] \to \mathbb{R}^d$, used as an approximation for the exact solution $\mathbf{u}$.

\subsection{Physics-Informed Neural Networks (PINNs)} In general, in the context of the Physics-Informed Neural Network, one can train the Neural Network $\mathcal{N}_\theta$ by minimizing the following loss function:
\begin{eqnarray}\label{loss1}
\mathcal{L}(\theta) = \omega_{eq} \mathcal{L}_{eq}(\theta) + \omega_{init}\mathcal{L}_{init}(\theta),
\end{eqnarray}
where
\[
\begin{split}
&\mathcal{L}_{eq}(\theta) := \frac{1}{N_r} \sum^{N_r}_{i=1} \|\mathcal{R}[\mathbf{u}_\theta](t_r^i)\|_2^2 = \frac{1}{N_r}\sum_{i=1}^{N_r} \|\partial_t \mathbf{u}_\theta(t_r^i) - \mathbf{F}(t^i_r,\mathbf{u}_\theta(t^i_r))\|_2^2, \qquad t_r^i\in \mathcal{T}_r,\\
&\mathcal{L}_{init}(\theta) := \|\mathcal{I}[\mathbf{u}_\theta]\|^2_{2} = \|\mathbf{u}_\theta(t_0) - \mathbf{u}_0\|^2_2,
\end{split}
\]
with $\mathcal{L}_{eq}(\theta)$, $\mathcal{L}_{init}(\theta)$ representing the residuals of the differential equation and the initial condition, respectively. Here
$\mathcal{T}_r := \{t_r^i\}^{N_r}_{i=1} \subset [t_0, T]$ denote the set of $N_r$ residual collocation points, and $\omega_{eq}$ and $\omega_{init}$ are the loss weighting coefficients.

\subsection{Sequential decomposition in time}\label{sec:sequential} 
Standard PINNs can be difficult to train over long time intervals, since that a network has to represent the entire path at once \cite{RaissiPerdikarisKarniadakis2019}. This is especially relevant in time-dependent problems, where the solution at later times depends on the quality of the approximation at earlier times \cite{WangSankaranPerdikaris2024}. A common way to reduce this difficulty is to divide the domain and train local neural networks, as in Conservative Physics-Informed Neural Network (cPINNs), Extended Physics-Informed Neural Network (XPINNs), and related parallel PINN methods \cite{JagtapKarniadakis2020, JagtapKharazmiKarniadakis2020, ShuklaHuJagtapKarniadakis2021}. However, our problem is a purely time-dependent initial value problem. Therefore, the main difficulty lies not in the spatial interfaces or the flow transmission conditions, but in the stable propagation of information from one time interval to the next. For this reason, we use a sequential time decomposition proposed in \cite{Roy}, where the interval is divided into consecutive subintervals, and a local neural network is trained on each of them. The approximation obtained at the end of a subinterval is used to initialize the next one. Thus, let
\[
t_0 = \tau_0 < \tau_1 < \cdots < \tau_M = T
\]
be a partition of $[t_0, T]$ and consider the set of subintervals 
\begin{equation}\label{eq:Im}
I_m = [\tau_{m-1}, \tau_{m}], \quad 1\leq m \leq M.
\end{equation}
Then, on each subinterval $[\tau_{m-1}, \tau_m]$, we consider a local neural network
\[
\mathcal{N}^{(m)}_{\theta_m} : I_m \to \mathbb{R}^d
\]
and a local trial function $\mathbf{u}_\theta^{(m)} : I_m \to \mathbb{R}^d$. The collection of trainable parameters is denoted by $\Theta := \{\theta_1, \dots,\theta_M\}$ and, for each $m = 1, \dots, M$, we denote the local residual collocation points set by $\mathcal{T}_r^{(m)} := \{t_{r,m}^i\}^{N_r^{(m)}}_{i=1} \subset I_m$. When the time interval is divided, the local approximations must agree at the connection points $\tau_1, \dots, \tau_{M-1}$. Otherwise, the global approximation may have jumps between consecutive intervals. We consider two ways to impose this connection: a soft-constrained formulation, where the mismatch is penalized in the loss function, and a hard-constrained formulation, where the connection is built into the trial function \cite{Roy}. This leads to consider the Sequential Soft-Constrained PINNs and the Hard-Constrained Sequential Constrained PINNs as described below. 
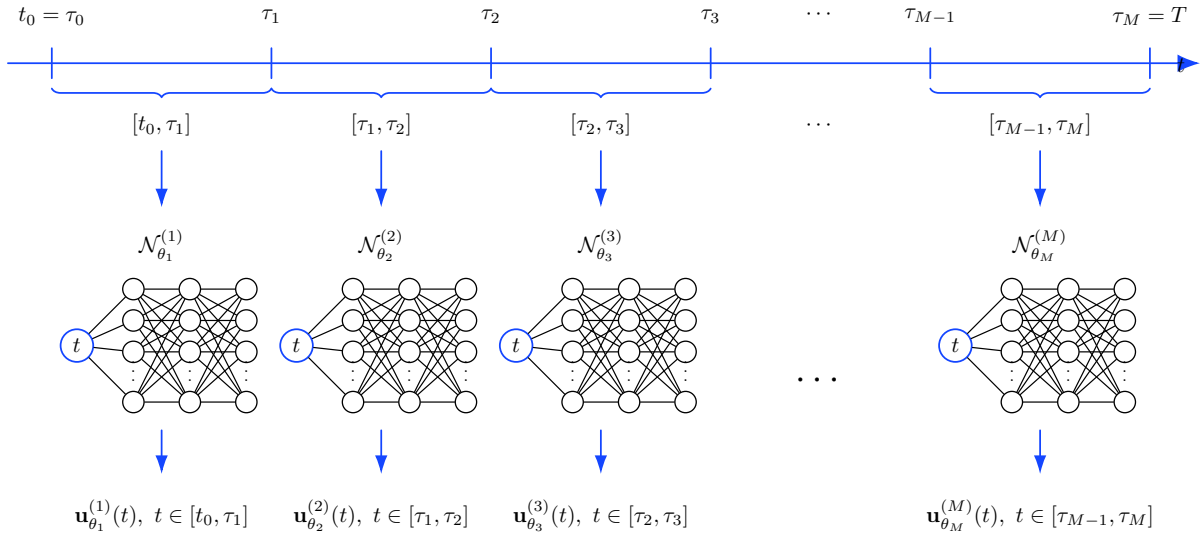
\begin{figure}[H]
\resizebox{\linewidth}{!}{
\begin{tikzpicture}[x=1cm,y=1cm,>=Latex,line cap=round,line join=round]
  \definecolor{blueacc}{RGB}{20,70,255}
  \tikzset{
    time/.style={blueacc, line width=0.8pt},
    downarr/.style={blueacc,-{Latex[length=3mm,width=2mm]}, line width=0.8pt},
    brace/.style={blueacc,decorate,decoration={brace,mirror,amplitude=4pt,raise=1.8pt},line width=0.8pt},
    intlabel/.style={font=\normalsize},
    netnode/.style={circle, draw=black, line width=0.6pt, fill=white, minimum size=0.34cm, inner sep=0pt},
    tnode/.style={circle, draw=blueacc, line width=0.8pt, fill=white, minimum size=0.52cm, inner sep=0pt},
    netline/.style={draw=black, line width=0.55pt},
  }

  % timeline
  \coordinate (T0) at (0,0);
  \coordinate (T1) at (3.5,0);
  \coordinate (T2) at (7.0,0);
  \coordinate (T3) at (10.5,0);
  \coordinate (TMm) at (14.0,0);
  \coordinate (TM) at (17.5,0);
  \draw[time,-{Latex[length=3.6mm,width=2.8mm]}] (-0.7,0) -- (18.3,0);
  \foreach \P/\lab in {T0/{t_0=\tau_0},T1/{\tau_1},T2/{\tau_2},T3/{\tau_3},TMm/{\tau_{M-1}},TM/{\tau_M=T}}{
     \draw[time] ($ (\P)+(0,-0.24) $) -- ($ (\P)+(0,0.24) $);
     \node[font=\normalsize] at ($ (\P)+(0,0.75) $) {$\lab$};
  }
  \node[font=\normalsize] at (12.25,0.78) {$\cdots$};
  \node[font=\normalsize] at (18.0,-0.03) {$t$};

  % braces and interval labels
  \foreach \A/\B/\txt in {T0/T1/{[t_0,\tau_1]},T1/T2/{[\tau_1,\tau_2]},T2/T3/{[\tau_2,\tau_3]},TMm/TM/{[\tau_{M-1},\tau_M]}}{
    \draw[brace] ($ (\A)+(0,-0.34) $) -- ($ (\B)+(0,-0.34) $);
    \node[intlabel] at ($ (\A)!0.5!(\B)+(0,-0.98) $) {$\txt$};
  }
  \node[font=\normalsize] at (12.25,-0.98) {$\cdots$};

  % positions for networks
  \coordinate (N1c) at (1.75,-5.25);
  \coordinate (N2c) at (5.25,-5.25);
  \coordinate (N3c) at (8.75,-5.25);
  \coordinate (NMc) at (15.75,-5.25);

  % arrows down + network labels
  \foreach \C/\lab in {N1c/{\mathcal{N}^{(1)}_{\theta_1}},N2c/{\mathcal{N}^{(2)}_{\theta_2}},N3c/{\mathcal{N}^{(3)}_{\theta_3}},NMc/{\mathcal{N}^{(M)}_{\theta_M}}}{
     \draw[downarr] ($ (\C)+(0,3.85) $) -- ($ (\C)+(0,2.95) $);
     \node[font=\normalsize] at ($ (\C)+(0,2.35) $) {$\lab$};
  }

  \node[font=\LARGE] at (12.25,-5.05) {$\cdots$};

  % network drawing macro
  \newcommand{\drawnet}[3][1]{%
    \begin{scope}[shift={(#2)}]
      \begin{scope}[scale=#1, xshift=-0.1cm, yshift=0.75cm]
      \node[tnode] (t) at (-1.25,0) {$t$};
      \foreach \y/\name in {0.9/a1,0.4/a2,-0.1/a3,-0.9/a4}{\node[netnode] (\name) at (-0.35,\y) {};}
      \foreach \y/\name in {0.9/b1,0.4/b2,-0.1/b3,-0.9/b4}{\node[netnode] (\name) at (0.55,\y) {};}
      \foreach \y/\name in {0.9/c1,0.4/c2,-0.1/c3,-0.9/c4}{\node[netnode] (\name) at (1.45,\y) {};}
      \node[font=\scriptsize] at (-0.35,-0.3) {$\vdots$};
      \node[font=\scriptsize] at (0.55,-0.3) {$\vdots$};
      \node[font=\scriptsize] at (1.45,-0.3) {$\vdots$};
      % input to first layer
      \foreach \n in {a1,a2,a3,a4}{\draw[netline] (t) -- (\n);}
      % adjacent layer full connections
      \foreach \i in {a1,a2,a3,a4}{\foreach \j in {b1,b2,b3,b4}{\draw[netline] (\i) -- (\j);}}
      \foreach \i in {b1,b2,b3,b4}{\foreach \j in {c1,c2,c3,c4}{\draw[netline] (\i) -- (\j);}}
      \draw[downarr] (0.1,-1.35) -- (0.1,-2.05);
      \node[font=\normalsize] at (0.1,-2.7) {$#3$};
      \end{scope}
    \end{scope}
  }

  \drawnet{N1c}{\mathbf{u}^{(1)}_{\theta_1}(t),\ t\in[t_0,\tau_1]}
  \drawnet{N2c}{\mathbf{u}^{(2)}_{\theta_2}(t),\ t\in[\tau_1,\tau_2]}
  \drawnet{N3c}{\mathbf{u}^{(3)}_{\theta_3}(t),\ t\in[\tau_2,\tau_3]}
  \drawnet{NMc}{\mathbf{u}^{(M)}_{\theta_M}(t),\ t\in[\tau_{M-1},\tau_M]}

\end{tikzpicture}
}
\caption{Time partition in sequential PINNs.}
\label{fig:sequentialPinns}
\end{figure}

\subsubsection{Soft-Constrained Sequential PINNs (SCS-PINNs)}
In the soft-constrained formulation, the values of the local networks at the interfaces are not imposed {\it a priori}. Following the treatment of interface conditions in domain-decomposition PINNs \cite{JagtapKarniadakis2020, JagtapKharazmiKarniadakis2020, Roy}, continuity across consecutive subintervals is promoted by adding suitable penalty terms to the loss function. A simple choice for $\mathbf{u}_\theta$ is 
\begin{equation}\label{eq:soft_NN}
\mathbf{u}_\theta(t) := \mathcal{N}_\theta (t).
\end{equation}
In this case, the initial condition is not satisfied {\it a priori} and therefore it must be incorporated into the optimization functional. Thus, the corresponding soft-constrained loss is the same loss function defined in (\ref{loss1}), that is, 
\[
\mathcal{L}^{soft} (\theta) = \omega_{eq}\mathcal{L}_{eq}(\theta) + \omega_{init}\mathcal{L}_{init}(\theta).
\]
Therefore, in the sequential soft-constrained formulation, temporal continuity is weakly enforced using the loss functional. For each connection point $\tau_m$, $m=2, \dots, M$, we define the local loss function
\[
\mathcal{L}_{init}^{(m)}(\theta_m) := \|\mathbf{u}_{\theta_m}^{(m)}(\tau_{m-1})-\mathbf{u}_{\theta_{m-1}}^{(m-1)}(\tau_{m-1})\|^2_2,
\]
and, for the initial condition on the first interval we consider
\[
\mathcal{L}_{init}^{(1)}(\theta_1) := \|\mathbf{u}_{\theta_1}^{(1)}(t_0) - \mathbf{u}_0\|^2_{2}.
\]
Thus, the soft-constrained sequential loss is
\[
    \mathcal{L}_{SCS}(\Theta) = \sum_{m=1}^M \omega_{eq}^{(m)} \mathcal{L}_{eq}^{(m)}(\theta_m) + \omega_{init} \mathcal{L}_{init}^{(1)}(\theta_1) + \sum_{m=2}^{M} \omega_{init}^{(m)} \mathcal{L}_{init}^{(m)}(\theta_m).
\]
In order to implement the SCS-PINNs methodology, we consider Algorithm \ref{Alg:SCSPINNs} established below; the terms in blue indicate the choices that will distinguish SCS-PINNs from HCS-PINNs, which are described in the next subsection.
% ============================================================
% Algorithm: SCS-PINNs
% ============================================================

\begin{breakablealgorithm}\label{Alg:SCSPINNs}
\caption{Using SCS-PINNs to approximate the solution of an ODE system}
\footnotesize
\begin{algorithmic}[1]

\algline{\textbf{Step 1: Define the dynamical system}}
\alglinei{\textbf{INPUT:}  Right-hand side function $\mathbf{F}$, initial condition $\mathbf{u}_0\in\mathbb{R}^d$, interval $[t_0,T]$, partition size $M$, number of collocation points and weights $\omega_{\mathrm{eq}}^{(m)}$, $\omega_{\mathrm{init}}$ and $\omega_{\mathrm{init}}^{(m)}$.}
\alglinei{Consider $\mathbf{u}'(t)=\mathbf{F}(t,\mathbf{u}(t))$, $t\in[t_0,T]$, with $\mathbf{u}(t_0)=\mathbf{u}_0$.}
\alglinei{Define $\mathcal{R}[\mathbf{v}](t):=\mathbf{v}'(t)-\mathbf{F}(t,\mathbf{v}(t))$ and $\mathcal{I}[\mathbf{v}]:=\mathbf{v}(t_0)-\mathbf{u}_0$.}

\algline{\textbf{Step 2: Decompose the time interval}}
\alglinei{Take a partition $t_0=\tau_0<\tau_1<\cdots<\tau_M=T$.}
\alglinei{For $m=1,\ldots,M$, set $I_m$ as \eqref{eq:Im}, and choose the collocation points $\mathcal{T}_r^{(m)}:=\{t_{r,m}^{i}\}_{i=1}^{N_r^{(m)}}\subset I_m$.}

\algline{\textbf{Step 3: Define the local approximations}}
\alglinei{For each $m=1,\ldots,M$, define a neural network $\mathcal{N}_{\theta_m}^{(m)}:I_m\to\mathbb{R}^d$.}
\alglinei{On each subinterval $I_m$, define {\color{blue}$\mathbf{u}_{\theta_m}^{(m)}(t):=\mathcal{N}_{\theta_m}^{(m)}(t)$}.}
\alglinei{The initial condition and the connection conditions are included in the loss function.}

\algline{\textbf{Step 4: Define the loss function}}
\alglinei{For the first neural network $\mathcal{N}_{\theta_1}^{(1)}$, define the loss}
\alglinei{{\color{blue}$\mathcal{L}_{SCS}^{(1)}(\theta_1):=\dfrac{\omega_{eq}^{(1)}}{N_r^{(1)}}\sum_{i=1}^{N_r^{(1)}}\left\|\mathcal{R}[\mathbf{u}_{\theta_1}^{(1)}](t_{r,1}^{i})\right\|_2^2+\omega_{\mathrm{init}}\left\|\mathbf{u}_{\theta_1}^{(1)}(t_0)-\mathbf{u}_0\right\|_2^2$.}}
\alglinei{For each $\mathcal{N}_{\theta_m}^{(m)}$ with $m=2,\ldots,M$, define the loss}
\alglinei{{\color{blue}$\mathcal{L}_{SCS}^{(m)}(\theta_m):=\dfrac{\omega_{eq}^{(m)}}{N_r^{(m)}}\sum_{i=1}^{N_r^{(m)}}\left\|\mathcal{R}[\mathbf{u}_{\theta_m}^{(m)}](t_{r,m}^{i})\right\|_2^2+\omega_{\mathrm{init}}^{(m)}\left\|\mathbf{u}_{\theta_m}^{(m)}(\tau_{m-1})-\mathbf{u}_{\theta_{m-1}^{*}}^{(m-1)}(\tau_{m-1})\right\|_2^2$.}}

\algline{\textbf{Step 5: Train the SCS-PINN model}}
\alglinei{Compute $\theta_1^{*}\approx \operatorname*{argmin}_{\theta_1}\mathcal{L}_{SCS}^{(1)}(\theta_1)$ using automatic differentiation and a gradient-based optimizer.}
\alglinei{For $m=2,\ldots,M$, compute $\theta_m^{*}\approx \operatorname*{argmin}_{\theta_m}\mathcal{L}_{SCS}^{(m)}(\theta_m)$ using the previously trained value $\mathbf{u}_{\theta_{m-1}^{*}}^{(m-1)}(\tau_{m-1})$.} 

\algline{\textbf{Step 6: Construct the global approximation}}
\alglinei{Define $\mathbf{u}_{\Theta^*}^{\mathrm{SCS}}(t):=\mathbf{u}_{\theta_m^{*}}^{(m)}(t)$ for $t\in I_m$, where $\Theta^*:=\{\theta_1^*,\ldots,\theta_M^*\}$ where $\Theta^*$ is the full set of trainable parameters.}
\alglinei{\textbf{OUTPUT:} Piecewise approximation $\mathbf{u}_{\Theta^*}^{\mathrm{SCS}}(t)\approx\mathbf{u}(t)$.}

\end{algorithmic}
\end{breakablealgorithm}

\subsubsection{Hard-Constrained Sequential PINNs (HCS-PINNs)}
The hard-constrained formulation follows the same sequential idea, but the connection with the previous interval is built into the trial function. This avoids adding a separate continuity penalty in the loss, as in the hard-constrained sequential PINN strategy of \cite{Roy}. A standard hard-constrained choice is
\[
\mathbf{u}_\theta (t) := \mathbf{u}_0 + \phi(t) \mathcal{N}_\theta (t),
\]
where $\phi : [t_0, T] \to \mathbb{R}$ is a smooth scalar function satisfying $\phi(t_0)=0$. For instance, one may take $\phi(t)=t-t_0$. Then,
\[
\mathbf{u}_\theta(t_0) = \mathbf{u}_0
\]
holds identically for every $\theta$. In this case, the initial condition is no longer part of the loss functional and the corresponding loss reduces to
\[
\mathcal{L}^{hard}(\theta) = \omega_{eq}\mathcal{L}_{eq}(\theta).
\]
Thus, for the sequential decomposition, the continuity between connection points is set directly into the local functions. On the first interval $[t_0, \tau_1]$, we define
\[
\mathbf{u}_{\theta_1}^{(1)}(t) = \mathbf{u}_0 + \phi_1(t) \mathcal{N}_{\theta_1}^{(1)}(t), \quad t\in I_1,
\]
and for $m \geq 2$, we define recursively
\[
\mathbf{u}_{\theta_m}^{(m)} (t)= \mathbf{u}_{\theta_{m-1}}^{(m-1)}(\tau_{m-1}) + \phi_m (t) \mathcal{N}_{\theta_m}^{(m)}(t), \qquad t \in I_m,
\]
where $\phi_m: I_m \to \mathbb{R}$ satisfies $\phi_m(\tau_{m-1}) = 0$. The corresponding sequential hard-constrained loss is given by 
\[
\mathcal{L}_{HCS}(\Theta) = \sum_{m=1}^M \omega_{eq}^{(m)}\mathcal{L}_{eq}^{(m)}(\theta_m).
\]
Therefore, the continuity of the solution and the initial condition are imposed directly. As implementation, we can use Algorithm \ref{Alg:HCSPINNs}; the terms in blue indicate the choices that distinguish HCS-PINNs from SCS-PINNs.

% ============================================================
% Algorithm: HCS-PINNs
% ============================================================

\begin{breakablealgorithm}\label{Alg:HCSPINNs}
\caption{Using HCS-PINNs to approximate the solution of an ODE system}
\footnotesize
\begin{algorithmic}[1]

\algline{\textbf{Step 1: Define the dynamical system}}
\alglinei{\textbf{INPUT:} Right-hand side function $\mathbf{F}$, initial condition $\mathbf{u}_0\in\mathbb{R}^d$, interval $[t_0,T]$, partition size $M$, number of collocation points and weights $\omega_{\mathrm{eq}}^{(m)}$.}
\alglinei{Consider $\mathbf{u}'(t)=\mathbf{F}(t,\mathbf{u}(t))$, $t\in[t_0,T]$, with $\mathbf{u}(t_0)=\mathbf{u}_0$.}
\alglinei{Define $\mathcal{R}[\mathbf{v}](t):=\mathbf{v}'(t)-\mathbf{F}(t,\mathbf{v}(t))$ and $\mathcal{I}[\mathbf{v}]:=\mathbf{v}(t_0)-\mathbf{u}_0$.}

\algline{\textbf{Step 2: Decompose the time interval}}
\alglinei{Take a partition $t_0=\tau_0<\tau_1<\cdots<\tau_M=T$.}
\alglinei{For $m=1,\ldots,M$, set $I_m$ as \eqref{eq:Im} and choose the collocation points $\mathcal{T}_r^{(m)}:=\{t_{r,m}^{i}\}_{i=1}^{N_r^{(m)}}\subset I_m$.}

\algline{\textbf{Step 3: Define the local approximations}}
\alglinei{For each $m=1,\ldots,M$, define a neural network $\mathcal{N}_{\theta_m}^{(m)}:I_m\to\mathbb{R}^d$.}
\alglinei{Take $\phi_m(t):=t-\tau_{m-1}$ on $I_m$, so that $\phi_m(\tau_{m-1})=0$.}
\alglinei{On $I_1$, define {\color{blue}$\mathbf{u}_{\theta_1}^{(1)}(t):=\mathbf{u}_0+\phi_1(t)\mathcal{N}_{\theta_1}^{(1)}(t)$}.}
\alglinei{For $m=2,\ldots,M$, on $I_m$, define {\color{blue}$\mathbf{u}_{\theta_m}^{(m)}(t):=\mathbf{u}_{\theta_{m-1}^{*}}^{(m-1)}(\tau_{m-1})+\phi_m(t)\mathcal{N}_{\theta_m}^{(m)}(t)$}.}
\algline{\textbf{Step 4: Define the loss function}}
\alglinei{For each $\mathcal{N}_{\theta_m}^{(m)}$ ($m = 1,\dots,M)$, define the loss {\color{blue}$\mathcal{L}_{HCS}^{(m)}(\theta_m):=\dfrac{\omega_{eq}^{(m)}}{N_r^{(m)}}\sum_{i=1}^{N_r^{(m)}}\left\|\mathcal{R}[\mathbf{u}_{\theta_m}^{(m)}](t_{r,m}^{i})\right\|_2^2$.}}

\algline{\textbf{Step 5: Train the HCS-PINN model}}
\alglinei{For each $m=1,\dots,M$, compute $\theta_m^{*}\approx\operatorname*{argmin}_{\theta_m}\mathcal{L}_{HCS}^{(m)}(\theta_m)$ using automatic differentiation and a gradient-based optimizer.}
\alglinei{After training on $I_m$, set $\mathbf{u}_{\theta_m^{*}}^{(m)}(\tau_m)$ as the initial value for the next subinterval.}

\algline{\textbf{Step 6: Construct the global approximation}}
\alglinei{Define $\mathbf{u}_{\Theta^*}^{\mathrm{HCS}}(t):=\mathbf{u}_{\theta_m^{*}}^{(m)}(t)$ for $t\in I_m$, where $\Theta^*:=\{\theta_1^*,\ldots,\theta_M^*\}$ where $\Theta^*$ is the full set of trainable parameters.}
\alglinei{\textbf{OUTPUT:} Piecewise continuous approximation $\mathbf{u}_{\Theta^*}^{\mathrm{HCS}}(t)\approx\mathbf{u}(t)$ satisfying the initial condition exactly.}

\end{algorithmic}
\end{breakablealgorithm}

\subsection{A generalization error estimate for Sequential PINNs approximations}
In this section, we present an estimate of the error associated with the sequential PINN approximations introduced above. The generalization error of the sequential approximation is measured by the norm
\[
\mathcal{E}_G(\Theta^*) := \left(\sum_{m=1}^M \int_{I_m}\|\mathbf{u}(t)-\mathbf{u}_{\theta^*_m}^{(m)}(t)\|^2_2dt\right)^{1/2},
\]
where $I_m$ denotes the $m$-th subinterval of the time partition introduced in Section \ref{sec:sequential}, and $\Theta^*$ is the full set of trainable parameters.
For each subinterval, we define the local residual by
\[
\mathbf{R}_m(t):=\partial_t\mathbf{u}_{\theta_m^*}^{(m)}(t)
-\mathbf{F}\left(t,\mathbf{u}_{\theta_m^*}^{(m)}(t)\right),
\qquad t\in I_m.
\]

\begin{theo}\label{teo:cotaContinua}
    Let $\mathbf{u}$ be the solution of the general ODE system \eqref{eq:EDO_abs} and let $\mathbf{u}_{\Theta^*}$ be the sequential approximation defined in Section \ref{sec:sequential}. Suppose that there exists a constant $L>0$ such that
    \[
    \|\mathbf{F}(t,\mathbf{v})-\mathbf{F}(t,\mathbf{w})\|_2
    \leq L\|\mathbf{v}-\mathbf{w}\|_2
    \]
    for every $t\in[t_0,T]$ and for every $\mathbf{v},\mathbf{w}$ in the range of the exact solution and the sequential approximation. Then,
    \[
    \begin{split}
    \mathcal{E}_G(\Theta^*) \leq \sqrt{T-t_0}&e^{L(T-t_0)}\Bigg[\|\mathbf{u}_0-\mathbf{u}^{(1)}_{\theta^*_1}(t_0)\|_2 + \sum_{j=1}^{M-1} \|\mathbf{u}_{\theta^*_{j+1}}^{(j+1)}(\tau_j) - \mathbf{u}_{\theta^*_j}^{(j)}(\tau_j)\|_2 \\ & +  \sqrt{T-t_0}\left(\sum_{m=1}^M \int_{I_m}\|\mathbf{R}_m(t)\|^2_2dt\right)^{1/2}\Bigg].
    \end{split}
    \]
    In particular, for HCS-PINNs, we obtain
\[
\mathcal{E}_G(\Theta^*)
\leq
(T-t_0)
e^{L(T-t_0)}
\left(
\sum_{m=1}^{M}
\int_{I_m}
\|\mathbf{R}_m(t)\|_2^2dt
\right)^{1/2}.
\]
\end{theo}
{
\begin{proof}
    For $t \in I_m$, we define the local error as $\mathbf{e}_m(t) := \mathbf{u}(t) - \mathbf{u}_{\theta^*_m}^{(m)}(t).$ In the subinterval $I_1$, the initial error satisfies
    \begin{equation}\label{eq:ErrorInicialt0}
    \|\mathbf{e}_1(t_0)\|_2 = \|\mathbf{u}_0 - \mathbf{u}_{\theta^*_{1}}^{(1)}(t_0)\|_2.
    \end{equation}
    Furthermore, at each connection point $\tau_j$, the error satisfies
    \begin{equation}\label{eq:EstimativoErrorInicialesB}
    \begin{split}
    \|\mathbf{e}_{j+1}(\tau_{j})\|_2 & = \|\mathbf{u}(\tau_j) - \mathbf{u}_{\theta^*_{j+1}}^{(j+1)}(\tau_j)\|_2 \\
    & \leq \|\mathbf{u}(\tau_j)-\mathbf{u}^{(j)}_{\theta^*_j}(\tau_j)\|_2 \hspace{0.1cm}+\hspace{0.1cm} \|\mathbf{u}_{\theta^*_j}^{(j)}(\tau_j) - \mathbf{u}_{\theta^*_{j+1}}^{(j+1)}(\tau_j)\|_2 \\
    & = \|\mathbf{e}_j(\tau_j)\|_2 \hspace{0.1cm}+\hspace{0.1cm} \|\mathbf{u}_{\theta^*_j}^{(j)}(\tau_j) - \mathbf{u}_{\theta^*_{j+1}}^{(j+1)}(\tau_j)\|_2.
    \end{split}
    \end{equation}
    On the other hand, deriving $\mathbf{e}_m(t)$ for $t\in I_m$ we obtain
    \[
    \mathbf{e}'_m(t) = \mathbf{F}(t,\mathbf{u}(t))-\mathbf{F}\left(t,\mathbf{u}_{\theta^*_m}^{(m)}(t)\right) - \mathbf{R}_m(t).
    \]
    Integrating on $[\tau_{m-1}, t] \subset I_m$, we get
    \[
        \mathbf{e}_m(t) = \mathbf{e}_m(\tau_{m-1}) + \int_{\tau_{m-1}}^t \left[\mathbf{F}(s,\mathbf{u}(s))- \mathbf{F}\left(s,\mathbf{u}_{\theta^*_{m}}^{(m)}(s)\right)\right] ds - \int_{\tau_{m-1}}^t \mathbf{R}_m(s)ds.
    \]
    Taking the euclidean norm, applying the triangle inequality, and using the Lipschitz condition above, we deduce that
    \[
    \|\mathbf{e}_m(t)\|_2\leq \|\mathbf{e}_m(\tau_{m-1})\|_2 + L\int_{\tau_{m-1}}^t \|\mathbf{e}_m(s)\|_2ds + \int_{\tau_{m-1}}^t \|\mathbf{R}_m(s)\|_2ds.
    \]
    Applying the Gronwall inequality, we obtain
    \begin{equation}\label{eq:EstimativoLocalB}
        \|\mathbf{e}_m(t)\|_2 \leq e^{L(t-\tau_{m-1})}\|\mathbf{e}_m(\tau_{m-1})\|_2 + \int_{\tau_{m-1}}^t e^{L(t-s)}\|\mathbf{R}_m(s)\|_2ds.
    \end{equation}
    
\noindent
For $m=1$, with $t\in I_1=[t_0,\tau_1]$, from \eqref{eq:ErrorInicialt0}, \eqref{eq:EstimativoLocalB} and using that $t_0 \leq s$, we have
\begin{equation}\label{eq:m1B}
\|\mathbf{e}_1(t)\|_2\hspace{0.1cm}
\leq\hspace{0.1cm}
e^{L(t-t_0)}
\left(
\|\mathbf{u}_0-\mathbf{u}_{\theta_1^*}^{(1)}(t_0)\|_2
+
\int_{t_0}^t
\|\mathbf{R}_1(s)\|_2ds
\right).
\end{equation}
Now, for $m=2$, we first estimate the error at the point $\tau_1$. Using the estimate \eqref{eq:EstimativoErrorInicialesB} with $j=1$ and the estimate \eqref{eq:m1B} at $t=\tau_1$, we obtain
\[
\begin{aligned}
\|\mathbf{e}_2(\tau_1)\|_2
\hspace{0.1cm}\leq & \hspace{0.1cm}
\|\mathbf{e}_1(\tau_1)\|_2
\hspace{0.1cm}+\hspace{0.1cm}
\|\mathbf{u}_{\theta_1^*}^{(1)}(\tau_1)
-
\mathbf{u}_{\theta_2^*}^{(2)}(\tau_1)\|_2 \\
\leq &\hspace{0.1cm} e^{L(\tau_1-t_0)}
\Bigg(
\|\mathbf{u}_0-\mathbf{u}_{\theta_1^*}^{(1)}(t_0)\|_2
\hspace{0.1cm}+\hspace{0.1cm}
\|\mathbf{u}_{\theta_1^*}^{(1)}(\tau_1)
-
\mathbf{u}_{\theta_2^*}^{(2)}(\tau_1)\|_2 \hspace{0.1cm} +\hspace{0.1cm}
\int_{t_0}^{\tau_1}
\|\mathbf{R}_1(s)\|_2ds
\Bigg).
\end{aligned}
\]
Therefore, using \eqref{eq:EstimativoLocalB} with $m=2$, noting that $t_0 \leq s$, for $t\in I_2=[\tau_1,\tau_2]$ we get 
\begin{equation}\label{eq:m2B}
\begin{aligned}
\|\mathbf{e}_2(t)\|_2
\hspace{0.1cm}\leq\hspace{0.1cm}
e^{L(t-t_0)}
\Bigg(
&
\|\mathbf{u}_0-\mathbf{u}_{\theta_1^*}^{(1)}(t_0)\|_2
\hspace{0.1cm}+\hspace{0.1cm}
\|\mathbf{u}_{\theta_1^*}^{(1)}(\tau_1)
-
\mathbf{u}_{\theta_2^*}^{(2)}(\tau_1)\|_2 +
\int_{t_0}^{\tau_1}
\|\mathbf{R}_1(s)\|_2ds
+
\int_{\tau_1}^{t}
\|\mathbf{R}_2(s)\|_2ds
\Bigg).
\end{aligned}
\end{equation}
For $m=3$, we proceed similarly. At the point $\tau_2$, from \eqref{eq:EstimativoErrorInicialesB} with $j=2$ and the estimate \eqref{eq:m2B} with $t = \tau_2$, we get
\[
\begin{aligned}
\|\mathbf{e}_3(\tau_2)\|_2
\hspace{0.1cm}\leq\hspace{0.1cm}
 e^{L(\tau_2-t_0)}
\Bigg(
&
\|\mathbf{u}_0-\mathbf{u}_{\theta_1^*}^{(1)}(t_0)\|_2
\hspace{0.1cm}+\hspace{0.1cm}
\|\mathbf{u}_{\theta_1^*}^{(1)}(\tau_1)
-
\mathbf{u}_{\theta_2^*}^{(2)}(\tau_1)\|_2 \hspace{0.1cm}+\hspace{0.1cm}
\|\mathbf{u}_{\theta_2^*}^{(2)}(\tau_2)
-
\mathbf{u}_{\theta_3^*}^{(3)}(\tau_2)\|_2  \\
&+
\int_{t_0}^{\tau_1}
\|\mathbf{R}_1(s)\|_2ds
+
\int_{\tau_1}^{\tau_2}
\|\mathbf{R}_2(s)\|_2ds
\Bigg).
\end{aligned}
\]
Consequently, applying \eqref{eq:EstimativoLocalB} with $m=3$, for $t\in I_3=[\tau_2,\tau_3]$ we obtain
\[
\begin{aligned}
\|\mathbf{e}_3(t)\|_2
\leq
e^{L(t-t_0)}
\Bigg(
&
\|\mathbf{u}_0-\mathbf{u}_{\theta_1^*}^{(1)}(t_0)\|_2
+
\|\mathbf{u}_{\theta_1^*}^{(1)}(\tau_1)
-
\mathbf{u}_{\theta_2^*}^{(2)}(\tau_1)\|_2  +
\|\mathbf{u}_{\theta_2^*}^{(2)}(\tau_2)
-
\mathbf{u}_{\theta_3^*}^{(3)}(\tau_2)\|_2  \\
&+
\int_{t_0}^{\tau_1}
\|\mathbf{R}_1(s)\|_2ds
+
\int_{\tau_1}^{\tau_2}
\|\mathbf{R}_2(s)\|_2ds
+
\int_{\tau_2}^{t}
\|\mathbf{R}_3(s)\|_2ds
\Bigg).
\end{aligned}
\]
Noting by finite induction, for every $t\in I_m$ it holds
\[
\begin{aligned}
\|\mathbf{e}_m(t)\|_2
\leq
e^{L(t-t_0)}
\Bigg(
&
\|\mathbf{u}_0-\mathbf{u}_{\theta_1^*}^{(1)}(t_0)\|_2 +
\sum_{j=1}^{m-1}
\|\mathbf{u}_{\theta_j^*}^{(j)}(\tau_j)
-
\mathbf{u}_{\theta_{j+1}^*}^{(j+1)}(\tau_j)\|_2  \\
&+
\sum_{i=1}^{m-1}
\int_{\tau_{i-1}}^{\tau_i}
\|\mathbf{R}_i(s)\|_2ds
+
\int_{\tau_{m-1}}^t
\|\mathbf{R}_m(s)\|_2ds
\Bigg).
\end{aligned}
\]
In particular, since that $m\leq M$ and $t\leq \tau_m\leq T$,
\[
\begin{aligned}
\|\mathbf{e}_m(t)\|_2
\leq
e^{L(T-t_0)}
\Bigg(
&
\|\mathbf{u}_0-\mathbf{u}_{\theta_1^*}^{(1)}(t_0)\|_2 +
\sum_{j=1}^{M-1}
\|\mathbf{u}_{\theta_j^*}^{(j)}(\tau_j)
-
\mathbf{u}_{\theta_{j+1}^*}^{(j+1)}(\tau_j)\|_2  +
\sum_{m=1}^{M}
\int_{I_m}
\|\mathbf{R}_m(s)\|_2ds
\Bigg),
\end{aligned}
\]
and thus, using the Cauchy-Schwarz inequality, we obtain
\[
\begin{aligned}
\|\mathbf{e}_m(t)\|_2
\leq
e^{L(T-t_0)}
\Bigg(
&
\|\mathbf{u}_0-\mathbf{u}_{\theta_1^*}^{(1)}(t_0)\|_2 +
\sum_{j=1}^{M-1}
\|\mathbf{u}_{\theta_j^*}^{(j)}(\tau_j)
-
\mathbf{u}_{\theta_{j+1}^*}^{(j+1)}(\tau_j)\|_2  \\ &+ \sqrt{T-t_0}\left(
\sum_{m=1}^{M}
\int_{I_m}
\|\mathbf{R}_m(s)\|^2_2ds\right)^{1/2}
\Bigg).
\end{aligned}
\]
Squaring, integrating over $I_m$ and summing from $m=1$ to $M$, we get
\[
\begin{aligned}
\sum_{m=1}^M
\int_{I_m}
\|\mathbf{e}_m(t)\|_2^2dt \leq 
(T-t_0)e^{2L(T-t_0)}
\Bigg(&
\|\mathbf{u}_0-\mathbf{u}_{\theta_1^*}^{(1)}(t_0)\|_2 
+
\sum_{j=1}^{M-1}
\|\mathbf{u}_{\theta_j^*}^{(j)}(\tau_j)
-
\mathbf{u}_{\theta_{j+1}^*}^{(j+1)}(\tau_j)\|_2
 \\
& +
\sqrt{T-t_0}\left(\sum_{m=1}^{M}
\int_{I_m}
\|\mathbf{R}_m(s)\|_2^2ds\right)^{1/2}
\Bigg)^2.
\end{aligned}
\]
Consequently,
\begin{equation}\label{eq:generalDem}
\begin{split}
\mathcal{E}_G(\Theta^*) = \left(\sum_{m=1}^M\int_{I_m}\|\mathbf{e}_m(t)\|^2_2dt\right)&^{1/2} \leq \sqrt{T-t_0}e^{L(T-t_0)}\Bigg[\|\mathbf{u}_0-\mathbf{u}^{(1)}_{\theta^*_1}(t_0)\|_2 \\ &+ \sum_{j=1}^{M-1} \|\mathbf{u}_{\theta^*_{j+1}}^{(j+1)}(\tau_j) - \mathbf{u}_{\theta^*_j}^{(j)}(\tau_j)\|_2 + \sqrt{T-t_0}\left(\sum_{m=1}^M \int_{I_m}\|\mathbf{R}_m(t)\|^2_2dt\right)^{1/2}\Bigg].
\end{split}
\end{equation}
For HCS-PINNs, the initial condition and the connection conditions between subintervals are imposed exactly. Hence,
\[
\mathbf{u}_{\theta_1^*}^{(1)}(t_0)=\mathbf{u}_0, \qquad \text{and}\qquad 
\mathbf{u}_{\theta_{j+1}^*}^{(j+1)}(\tau_j)
=
\mathbf{u}_{\theta_j^*}^{(j)}(\tau_j),
\qquad j=1,\ldots,M-1.
\]
Therefore, from \eqref{eq:generalDem} we directly obtain
\[
\mathcal{E}_G(\Theta^*)
\leq
(T-t_0)e^{L(T-t_0)}
\left(
\sum_{m=1}^{M}
\int_{I_m}
\|\mathbf{R}_m(t)\|_2^2dt
\right)^{1/2}.
\]
\end{proof}
}
Since the training error depends on the collocation points, we introduce the quadrature formulation for the numerical approximation of the integral of a function $g(t)$ on $I_m$ as
\[
\left|\int_{I_m} g(t)dt - \sum_{j=1}^{N_r^{(m)}}q_{j,m} g(t_{r,m}^j) \right|\leq \varepsilon^{quad}_m,
\]
for $1\leq m\leq M$, where $t_{r,m}^j \in I_m$ are collocation points, $q_{j,m} > 0$ are quadrature weights, and $\varepsilon_m^{quad}$ is the quadrature error. Therefore, we derive the following corollary.
\begin{coro}
Assume that the hypotheses of Theorem \ref{teo:cotaContinua} are true. Let $\mathcal{T}_{r}^{(m)} = \{t_{r,m}^j\}_{j=1}^{N_r^{(m)}} \subset I_m$ be the residual collocation points used in the subinterval $I_m$. Then, the generalization error satisfies
    \[
    \begin{split}
    \mathcal{E}_G(\Theta^*) \leq \sqrt{T-t_0}&e^{L(T-t_0)}\Bigg[\|\mathbf{u}_0-\mathbf{u}^{(1)}_{\theta^*_1}(t_0)\|_2 + \sum_{j=1}^{M-1} \|\mathbf{u}_{\theta^*_{j+1}}^{(j+1)}(\tau_j) - \mathbf{u}_{\theta^*_j}^{(j)}(\tau_j)\|_2 \\ & +  \sqrt{T-t_0}\left(\sum_{m=1}^M \sum_{j=1}^{N_r^{(m)}}q_{j,m}\|\mathbf{R}_m(t_{r,m}^j)\|^2_2 + \sum_{m=1}^M\varepsilon_m^{quad}\right)^{1/2}\Bigg].
    \end{split}
    \]
    In particular, for HCS-PINNs, we obtain
\[
\mathcal{E}_G(\Theta^*)
\leq
(T-t_0)
e^{L(T-t_0)}
\left(
\sum_{m=1}^M\sum_{j=1}^{N_r^{(m)}}q_{j,m}\|\mathbf{R}_m(t_{r,m}^j)\|^2_2 + \sum_{m=1}^M\varepsilon_m^{quad}
\right)^{1/2}.
\]
\end{coro}
\section{Numerical simulations}

\subsection{Optimal Control Algorithm using PINNs}
To approximate numerically the optimal control problem, we use a gradient descent algorithm with Armijo line search. In each iteration, the state and adjoint systems are approximated separately by PINNs, including SCS-PINNs and HCS-PINNs. The incorporation of Armijo line search is motivated by the need to reduce the impact of the reduced gradient approximation error used in updating the controls, an error caused by the approximation error of the state-adjoint system using PINNs. Thus, this search criterion helps to obtain a more regular evolution of the objective functional, causing that the stopping criterion based on its stabilization be more effective. Thus, we propose the following iterative scheme:

\begin{breakablealgorithm}
\caption{Optimal control algorithm using separated state--adjoint PINN approximations}
\label{alg:CO}
\footnotesize
\begin{algorithmic}[1]

\algline{\textbf{Input:} Initial controls $[\ControlC^0,\ControlS^0]\in\mathcal{U}_{ad}$, reduction parameter $\beta\in(0,1)$, Armijo parameter $\eta\in(0,1)$, tolerance \texttt{tol}, positive integer $\mathcal{N}$, maximum number of iterations \texttt{maxIter}, and maximum number of Armijo trials \texttt{maxArmijo}.}

\algline{\textbf{Initialization:} Set $k=0$, choose $[\ControlC^0,\ControlS^0]\in\mathcal{U}_{ad}$.}
\algline{\textbf{Step 1: Compute of the state--adjoint approximation}}
\alglinei{{\textbf{Step 1.1: }}Given the current controls $[\ControlC^k,\ControlS^k]\in\mathcal{U}_{ad}$, compute the state approximation}
\alglineii{$S_{\theta^k}^k(t)=\bigl(u_{\theta^k}^k(t),\sigma_{\theta^k}^k(t)\bigr)$ by using (SCS or HCS)-PINN associated with the controlled state system \eqref{eq:SistemaControlado}.}

    \alglinei{\textbf{Step 1.2: } Using $[\ControlC^k,\ControlS^k]$ and $S_{\theta^k}^k$, compute the adjoint approximation $A_{\phi^k}^k(t)=\bigl(p_{1,\phi^k}^k(t),p_{2,\phi^k}^k(t)\bigr)$ by using (SCS or HCS)-PINN associated with the adjoint system \eqref{eq:adjoint_system_ode}.}
\algline{\textbf{Step 2: Evaluate of the current cost}}

\alglinei{Compute $J_k=
\widehat{J}\bigl([\ControlC^k,\ControlS^k]\bigr)$ using the state variables obtained from \textbf{Step 1}.}

\algline{\textbf{Step 3. Stop criterion}}
\alglinei{If $k > \mathcal{N}$, define $\delta_k = |J_k-J_{k-1}|.$
Given a tolerance $\texttt{tol} > 0$ and an integer $\mathcal{N} \in \mathbb{N}$, the algorithm stops if \[\delta_{k-j} < \texttt{tol},\qquad j =0,...,\mathcal{N}-1,\] 
that is, if the value of the functional remains stable during $\mathcal{N}$ consecutive iterations.}
\algline{\textbf{Step 4: Compute of the reduced gradient}}
\alglinei{Compute the reduced gradients components $d_{\ControlC}^k(t)=k_3(t)-\kappa u^k_{\theta^k}(t)\sigma^k_{\theta^k}(t)p_{1,\phi^k}^k(t),$ and\\ $d_{\ControlS}^k(t)=k_4(t)-S_c u_{\theta^k}^k(t)p_{2,\phi^k}^k(t).$ Then, define the reduced gradient as $[d_{\ControlC}^k(t), d_{\ControlS}^k(t)]^{\mathsf{T}}=\nabla \widehat{J}([\ControlC^k, \ControlS^k]).$}

\algline{\textbf{Step 5: Compute gradient descent + Armijo line search with projection in the admissible controls set}}

\alglinei{Compute the control for the next iteration as: $[\ControlC^{k+1}, \ControlS^{k+1}]^{\mathsf{T}} = Proj_{\mathcal{U}_{ad}}([\ControlC^k, \ControlS^k]^{\mathsf{T}} - \gamma([\ControlC^k, \ControlS^k])\nabla \widehat{J}([\ControlC^k, \ControlS^k])),$}
\alglinei{where $\gamma([\ControlC^k, \ControlS^k])$ is the Armijo step size, defined by an approximation line minimization in the following way (see \cite{Polak}): Given constants $\eta, \beta \in (0,1)$ independent of $[\ControlC^k, \ControlS^k]$, define the integer $j([\ControlC^k, \ControlS^k])$ by}
\algline{\[j([\ControlC^k, \ControlS^k]) := \min\Big\{j=0,1,...,\texttt{maxArmijo} : \widehat{J}\Big(Proj_{\mathcal{U}_{ad}}([\ControlC^k, \ControlS^k]^{\mathsf{T}}-\beta^j\nabla \widehat{J}([\ControlC^k, \ControlS^k]))\Big) -\widehat{J}\Big([\ControlC^k, \ControlS^k]\Big) \leq -\eta \beta ^j \|\nabla \widehat{J}([\ControlC^k, \ControlS^k])\|^2\Big\},\]}
\alglinei{where each evaluation of $\widehat{J}$ in this line search requires a new (SCS or HCS)-PINN approximation of the state system associated with the candidate controls, and define $\gamma([\ControlC^k, \ControlS^k])=\beta^{j([\ControlC^k, \ControlS^k])}$. Additionally,  
following the Algorithm \ref{alg:Projection} below, compute the projection on the admissible controls set, denoted by $Proj_{\mathcal{U}_{ad}}$.}

\vspace{0.03cm}
\algline{\textbf{Step 6: Advance the next iteration}}
\alglinei{If $k = \texttt{maxIter}$, then stop, else set $k=k+1$ and repeat this procedure.}
\end{algorithmic}
\end{breakablealgorithm}
\vspace{0.3cm}
For the step of projection onto the admissible controls set $\mathcal{U}_{ad}$ that appears in the Algorithm \ref{alg:CO}, we use the Algorithm \ref{alg:Projection}. This projection is done separately for each control. For the antiangiogenic therapy $\ControlS$, the projection is obtained by the truncation of the point over the interval $[0,1]$. For the cytotoxic therapy $\ControlC$, the integral restriction is handled using Proposition \ref{prop:proyeccion}.

\begin{breakablealgorithm}
\caption{Projection of controls onto the admissible controls set}
\label{alg:Projection}
\small
\begin{algorithmic}[1]

\algline{\textbf{Input:} Controls $\widetilde{\ControlC}(t)$, $\widetilde{\ControlS}(t)$, integral bound $c_{\max}>0$, final time $T$ and bisection tolerance \texttt{tolBis}.}

\algline{\textbf{Step 1. Project the antiangiogenic control $\widetilde{\ControlS}$.}}
\alglinei{Project $\widetilde{\ControlS}$ onto the pointwise admissible interval $[0,1]$ by $\ControlS(t) = \min\left\{\max\left\{\widetilde{\ControlS}(t),0 \right\},1 \right\}.$}

\algline{\textbf{Step 2. Project the chemotherapy control $\widetilde\ControlC$.}}
\alglinei{\textbf{Step 2.1. The $\gamma$ function}}
\alglineii{First, define:}
\alglinei{\[\phi(\gamma)=\int_0^T\min\{1,\max\{0,\widetilde{\ControlC}(t)-\gamma \}\}\,dt-c_{\max}.\]}
\alglinei{\textbf{Step 2.2. Set the bisection interval}}
\alglineii{If $\phi(0) \leq 0$, then stopping the algorithm and define $\ControlC(t) = \min\left\{\max\left\{\widetilde{\ControlC}(t) \right\},1 \right\}$. Else, define $\gamma_{low}$ as $0$ and search $\gamma_{high}$ such that $\phi(\gamma_{high}) < 0$. One way to find $\gamma_{high}$ is by taking $\gamma_{high}=1$ and doubling it until $\phi(\gamma_{high}) < 0.$}
\alglinei{\textbf{Step 2.3. Apply the bisection algorithm}}

\alglineii{While $|\gamma_{high}-\gamma_{low}|> \texttt{tolBis}$, define $\gamma_m=\frac{\gamma_{low}+\gamma_{high}}{2}$; if $\phi(\gamma_{low})\phi(\gamma_m)\leq0$, take $\gamma_{high}=\gamma_m$; if not, take $\gamma_{low}=\gamma_m$. At the end, consider $\gamma_m$.}

\alglinei{\textbf{Step 2.4. Apply the $\gamma_m$}}
\alglineii{Define $\ControlC(t) = \min\left\{\max\left\{\widetilde{\ControlC}(t)-\gamma_m,0 \right\},1 \right\}$.}

\algline{\textbf{Output:} Projected controls $[\ControlC,\ControlS]\in\mathcal{U}_{ad}$.}
\\
\end{algorithmic}
\end{breakablealgorithm}

\begin{remark}
In the implementation, when a step size is accepted in the Armijo search, the state solution and the value of the reduced functional are already available. We keep these values and use them in the next iteration, instead of solving the same state problem again with PINNs. This makes the code faster by avoiding a training step that was already done, but it does not change the projected gradient method.
\end{remark}

\subsection{Computational setup}
In all numerical experiments, the final time is fixed as $T=2$. The initial condition is set as
\[ u(0)=0.6, \qquad \sigma(0)=1.0, \]
and the desired values in the objective functional $J$ defined in (\ref{funcionalObjetivo}) are $\sigma_d(t) = 1.0$ for $t \in [0,T]$ and $\sigma_{dT} = 1.0$ in $t=T$. Likewise, we use $c_{\max} = 0.75$ as the constrained total amount for the cytotoxic therapy. The parameters for the ODE are on the Table \ref{tab:Parametros}, and the weights for the objective functional, $k_i$ and $l_j$ for $i=1,\dots,4$ and $j=1,2$, are on the Table \ref{tab:pesosFuncional}. In particular, the weights $l_1,l_2$ are taken greater than the corresponding ones $k_1,k_2,k_3,k_4,$ considering that is important to expect a minimum of the tumor density at the final time $T$. For the descent gradient algorithm with Armijo line search, we use the parameters given in the Table \ref{tab:ParametrosArmijo}.
\begin{table}[H]
\centering

\begin{minipage}{0.51\textwidth}
\centering
\begin{tabular}{cc|cc}
\hline
\multicolumn{2}{c|}{\textbf{Tumor equation}} & \multicolumn{2}{c}{\textbf{Oxygen equation}} \\ \hline
\textbf{Parameter}      & \textbf{Value}     & \textbf{Parameter}      & \textbf{Value}     \\ \hline
$\alpha$                & $4$                & $A_{ox}$                & $0.15$              \\
$\hat \rho$            & $3.5$              & $k_{ox}$                & $0.1$              \\
$b$                     & $1$                & $\beta$                & $1$                \\
$\kappa$                  & $2$               & $P_{er}S_v$                   & $0.1$                \\
                &                 & $S_c$                   & $0.4$              \\ \hline
\end{tabular}
\caption{Parameters for the state and adjoint systems.}
\label{tab:Parametros}
\end{minipage}
\hspace{0.05\textwidth}
\begin{minipage}{0.39\textwidth}
\centering
\begin{tabular}{c|c}
\hline
\textbf{Weight} & \textbf{Value} \\ \hline
$k_1$ & $1$ \\
$k_2$ & $1$ \\
$k_3$ & $1$ \\
$k_4$ & $1$ \\
$l_1$ & $10$ \\
$l_2$ & $10$ \\ \hline
\end{tabular}
\caption{Weights in the cost functional.}
\label{tab:pesosFuncional}
\end{minipage}
\end{table}
\begin{table}[H]
\centering
\par\vspace{0.5cm}
\begin{minipage}{0.5\textwidth}
\centering
\begin{tabular}{c|c}
\hline
\textbf{Parameter} & \textbf{Value}\\ \hline
Maximum control iterations & $800$ \\
Initial step size & $1$ \\
Minimum step size & $10^{-4}$ \\
Step reduction factor & $0.5$ \\
Maximum reductions & $2$ \\
Armijo constant & $10^{-3}$ \\
Stopping tolerance & $10^{-6}$ \\ 
Stabilized iterations ($\mathcal{N}$) & 3 \\ \hline
\end{tabular}
\caption{Gradient-descent and Armijo line-search parameters used in the control update.}
\label{tab:ParametrosArmijo}
\end{minipage}
\end{table}
To validate the effectiveness of (SCS or HCS)-PINNs for approximating the optimal control problem, we compare the results with those obtained via a classical solver.  In the classical approach, we use a temporal discretization with \(n_{\mathrm{Iter}}=250\) subintervals. This uniform time mesh is used in all the computations, namely, to represent the discrete controls, to evaluate the integrals of the cost functional by the trapezoidal rule, and to compare the PINN and the classical approximations. We use the fixed hyperparameters in all simulations given in the Table \ref{tab:ParametrosRedesFijos}. For the implementation we use Python, in particular the \texttt{PyTorch} framework for the construction, automatic differentiation, and training of PINNs, and \texttt{SciPy} for the classical reference solution, where the state and adjoint systems are integrated with the implicit Radau method with tolerances $\mathrm{rtol}=10^{-8}$ and $\mathrm{atol}=10^{-12}$. All computations are performed in double precision with a fixed random seed. In each subinterval, the collocation points are taken equispaced and, after training, we retain the parameters attaining the lowest loss along the epochs. We also point out that, at every iteration $k$ of Algorithm \ref{alg:CO}, the networks associated with the state and the adjoint systems are initialized again and trained from scratch with the whole number of epochs, thus, no information of the previous iteration is transferred to the next one. This explains the computational times reported below. To ensure experimental comparability, all simulations were run under the same hardware and software conditions on a single computer: a MacBook Air M4 (2025) with 8 GB of RAM.

\begin{table}[H]
\centering
\begin{tabular}{c|c}
\hline
\textbf{Hyperparameter} & \textbf{Fixed value} \\ \hline
Collocation points per subinterval (\(N^{(m)}_r\)) & \(250\) \\
Optimizer & Adam \\
Learning rate & \(10^{-3}\) \\
Activation function & tanh \\
SCS penalty weight $\omega_{init}^{(m)}$ (when is applicable) & \(5.0\) \\
Total epochs & $40\,000$ \\ \hline
\end{tabular}
\caption{Fixed PINN hyperparameters in all simulations.}
\label{tab:ParametrosRedesFijos}
\end{table}

\subsection{Experiment 1: Separated PINNs for states and adjoint }
In this subsection, we implement the Algorithm \ref{alg:CO} using separated state-adjoint PINNs. The PINN hyperparameters used in this experiment are given in Table \ref{tab:pinn_seq_vs_nonseq}.
\begin{table}[H]
\centering
\begin{minipage}{0.48\textwidth}
\centering
\textbf{Sequential PINN}
\vspace{1mm}

\begin{tabular}{c|c|c}
\hline
\textbf{Hyperparameter} & \textbf{States} & \textbf{Adjoint} \\ \hline
Hidden layers                        & 4      & 5      \\
Neurons per hidden layer             & 100    & 150    \\
Subintervals                         & 4      & 4      \\
\shortstack[c]{Collocation points\\per subinterval} & 250 & 250 \\
Epochs per subinterval               & 10\,000  & 10\,000  \\ \hline
\end{tabular}
\end{minipage}
\hfill
\begin{minipage}{0.48\textwidth}
\centering
\textbf{Non-sequential PINN}
\vspace{1mm}

\begin{tabular}{c|c|c}
\hline
\textbf{Hyperparameter} & \textbf{States} & \textbf{Adjoint} \\ \hline
Hidden layers                        & 4      & 5      \\
Neurons per hidden layer             & 100    & 150    \\
Subintervals                         & 1      & 1      \\
\shortstack[c]{Collocation points\\per subinterval} & 250 & 250 \\
Epochs per subinterval               & 40\,000  & 40\,000  \\ \hline
\end{tabular}
\end{minipage}
\caption{Hyperparameter comparison between sequential and non-sequential PINN configurations.}
\label{tab:pinn_seq_vs_nonseq}
\end{table}

\begin{remark}
The selected values ensure a fair computational comparison, both settings use the same cost training per network. In particular, the total number of optimizer updates is
\[
\mphrase{\text{optimizer}\\\text{updates}}
=
\bigl(\mphrase{\text{subintervals}}\bigr)
\mathbin{\times}
\bigg(\mphrase{\text{epochs per}\\\text{subinterval}}\bigg),
\]
settings satisfy $4\times 10{,}000 = 1\times 40{,}000.$ Likewise, the total number of residual evaluations is
\[
\mphrase{\text{residual}\\\text{evaluations}}
=
\bigl(\mphrase{\text{subintervals}}\bigr)
\mathbin{\times}
\bigg(\mphrase{\text{epochs per}\\\text{subinterval}}\bigg)
\mathbin{\times}
\bigg(\mphrase{\text{collocation points}\\\text{per subinterval}}\bigg).
\]
Then, $4\times 10\,000\times 250 = 1\times 40\,000\times 250$. Therefore, differences in the results are attributable to the time decomposition strategy rather than to an imbalance in the cost training per network. However, keeping the same number of collocation points per subinterval equalizes the training cost but not the density of collocation points; in the sequential setting, $N_r^{(m)}=250$ points are placed in each subinterval of length $T/4$, so the local density is four times larger than in the non-sequential setting. Hence, in the comparison below, the effect of the time decomposition cannot be completely separated from this refinement of the residual sampling.
\end{remark}
For this experiment, we evaluate four configurations obtained by combining two PINN formulations (SCS and HCS) with two temporal decompositions (4 subintervals and 1 subinterval) for both the state and adjoint systems. The SCS-PINN configuration with 1 subinterval corresponds to the original PINN formulation.\\

Table \ref{tab:results_time} shows, for each configuration, the approximation of the objective functional $J$, the stopping iteration of the PINN algorithm, and the total computation time. The reference value of $J$ obtained with the classical method is included in the last row. In general, the HCS configurations give values of $J$ closer to the classical reference, especially when 4 subintervals are used, while the SCS configurations have the shortest execution time.
\begin{table}[H]
\centering
\begin{tabular}{l|c|c|c}
\hline
\textbf{Configuration}     & \textbf{J Approximation} & \textbf{Iter Stop}                      & \textbf{Time} \\ \hline
SCS (4 subintervals)       & $15.878714$              & $24$                                         & $8\,102$ s  \\
SCS (1 subinterval)        & $15.623673$              & $56$                                         & $19\,529$ s    \\
HCS (4 subintervals)       & $15.450080$              & $79$                                        & $24\,773$ s \\
HCS (1 subinterval)        & $15.561244$              & $60$                                         & $20\,426$ s \\ \hline
\textbf{Classic reference} & $15.350898$              &   $118$                                           &               \\ \cline{1-4}
\end{tabular}
\caption{Performance comparison across SCS and HCS PINN configurations under sequential (4 subintervals) and non-sequential (1 subinterval) settings.}
\label{tab:results_time}
\end{table}

In order to quantify the comparison of the approximations, Table \ref{tab:results_L2} reports the errors of the optimal controls and the associated optimal states, measured in the $L^2(0,T)$-norm with respect to the classical reference, which is the same for the four configurations. These values are used in the discussion of the figures below.
\begin{table}[H]
\centering
\begin{tabular}{l|c|c|c|c}
\hline
\textbf{Configuration} & $\|\ControlC-\ControlC_{\mathrm{ref}}\|_{L^2}$ & $\|\ControlS-\ControlS_{\mathrm{ref}}\|_{L^2}$ & $\|u-u_{\mathrm{ref}}\|_{L^2}$ & $\|\sigma-\sigma_{\mathrm{ref}}\|_{L^2}$ \\ \hline
SCS (4 subintervals) & $0.1027$ & $0.3802$ & $0.1389$ & $0.0317$ \\
SCS (1 subinterval)  & $0.0242$ & $0.2523$ & $0.0103$ & $0.0245$ \\
HCS (4 subintervals) & $0.0255$ & $0.1135$ & $0.0138$ & $0.0112$ \\
HCS (1 subinterval)  & $0.0266$ & $0.1947$ & $0.0128$ & $0.0186$ \\ \hline
\end{tabular}

\caption{Errors of the optimal controls and optimal states in the $L^2(0,T)$-norm, with respect to the classical reference.}
\label{tab:results_L2}
\end{table}

Figure \ref{fig:sc_panel_global} shows the optimal controls obtained with Soft-Constrained PINNs, using the sequential and non-sequential methods. Both methods capture the main shape of the reference controls and identify the regions where the therapies are active or inactive. However, for the cytotoxic control $\ControlC$ (Figures \ref{fig:sc_control_c} and \ref{fig:sc_error_c}), the largest errors appear near the abrupt changes of the control. For the antiangiogenic control $\ControlS$ (Figures \ref{fig:sc_control_s} and \ref{fig:sc_error_s}), the non-sequential method gives a better approximation of the final activation region. In fact, in the Soft-Constrained case the non-sequential scheme is more accurate than the sequential one for both controls and both states (see Table \ref{tab:results_L2}). For the sequential approximation, the error is negligible in the first two subintervals and appears from the third one on, that is, once the cytotoxic control becomes active.

\input{graficas/control_SC_PINNs_exp1}

Figure \ref{fig:sc_states_panel} illustrates the optimal states obtained using the controls from Figures \ref{fig:sc_control_c} and \ref{fig:sc_control_s}. These states were also computed by means of an SC-PINN approximation. The graphics show that the PINN methods give a good approximation of the optimal states using the computed optimal controls, in comparison with the reference optimal states. In particular, the errors observed in Figures \ref{fig:sc_states_error_u} and \ref{fig:sc_states_error_sigma} are consistent with the errors reported in the approximation of the optimal controls. The same behavior described for the controls is observed in the states, in particular, the non-sequential scheme is more accurate, and the difference is particularly large for the tumor cell density $u$ (see Table \ref{tab:results_L2}).\\
\input{graficas/estados_SC_PINNs_exp1}

Figure \ref{fig:sc_functional_panel} shows the evolution of the cost functional for the approximations corresponding to SC-PINNs. It can be observed that in both cases, it presents a decreasing trend and subsequently stabilizes, approaching to the reference value, showing good performance of the Gradient Descent algorithm with Armijo line search.\\
\input{graficas/funcional_SC_PINNs_exp1}

Figure \ref{fig:hc_panel_global} shows the graphs corresponding to the optimal controls obtained with Hard-Constrained PINNs (sequential and non-sequential). In these graphs, an improvement in both controls can be observed compared with the Soft-Constrained case. The errors of the controls with respect to the classical reference are more localized and have smaller magnitude. In this case, a significant improvement is observed for the sequential case compared with the non-sequential case. In this case, the improvement of the sequential scheme compared to the non-sequential scheme is not uniform in both controls; in particular, it is substantial for the antiangiogenic control \ControlS, whose $L^2$-error is reduced by approximately $40\%$, and is concentrated in the final activation region for the cytotoxic control \ControlC, while the overall $L^2$-error of $\ControlC$ is practically the same in both schemes (see Table \ref{tab:results_L2}). \\

\input{graficas/control_HC_PINNs_exp1}

    The graphs corresponding to the states for the Hard-Constrained method (see Figure \ref{fig:hc_states_panel}) confirm this improvement. The approximated curves of the states $u$ and $\sigma$ (see Figures \ref{fig:hc_states_u} and \ref{fig:hc_states_sigma}) show a better approximation to the reference, and the errors (see Figures \ref{fig:hc_states_error_u} and \ref{fig:hc_states_error_sigma}) remain small during most of the interval, especially for the sequential scheme. Here, the two schemes behave differently for each state: the sequential scheme reduces the error in oxygen concentration $\sigma$ by approximately $40\%$, while for tumor cell density $u$, both schemes achieve comparable accuracy, with a slight advantage for the non-sequential scheme (see Table \ref{tab:results_L2}). This agrees with the behavior of the controls, since $\sigma$ is the state directly affected by the antiangiogenic control $\ControlS$.\\

\input{graficas/estados_HC_PINNs_exp1}

Finally, the graphs of the cost functional in the Hard-Constrained case (see Figure \ref{fig:hc_functional_panel}) show that both schemes (sequential and non-sequential) stay very close to the classical reference. Therefore, these results indicate that for this control problem and algorithm used, the Hard-Constrained formulation is more robust than the Soft-Constrained formulation in terms of accuracy of the functional, approximation of the states, and approximation of the optimal controls. In particular, for the Hard-Constrained case, it can be observed that the sequential scheme helps to obtain a better approximation of the objective functional, of the antiangiogenic control and of the oxygen concentration, while for the cytotoxic control and the tumor cell density both schemes are comparable, although with a higher computational cost. It is also worth noting that the time decomposition is beneficial only in the Hard-Constrained formulation; in the Soft-Constrained one, it degrades the approximation.

\input{graficas/funcional_HC_PINNs_exp1}

\subsection{Experiment 2: robustness with respect to initial controls}
\label{sec:experiment2_robustness}
This experiment tests the robustness with respect to the initial controls. For this, the test uses three different admissible initial controls and solves again the same problem, using the HCS-PINN formulation with 4 subintervals, which, according to the previous experiment, gives the best approximation. The idea is not to compare again the PINN formulations, but to observe if the final control keeps the same main behavior when the starting point is changed.\\

Figure \ref{fig:exp2_control_robustness_panel} contains the controls used in this test. Figure \ref{fig:exp2_initial_control_c} shows the three initial controls for the cytotoxic control $\ControlC$, and Figure \ref{fig:exp2_optimal_control_c} shows the optimal controls obtained after the minimization process. Figure \ref{fig:exp2_initial_control_s} gives the initial information for the antiangiogenic control $\ControlS$, and Figure \ref{fig:exp2_optimal_control_s} gives the optimal information for the same control. The initial controls are not close in the three cases. The first one is constant, the second one is concentrated in a part of the interval, and the last one has a more variable shape. Thus, if the method gives similar final controls, this is a good indication that the computation does not follow only the initial guess.\\

\input{graficas/exp2}

Figure \ref{fig:exp2_control_robustness_panel} shows that the optimal control has a similar structure in the three cases. For the cytotoxic therapy $\ControlC$, the control remains close to zero in the first part of the time interval and later becomes active near the final part. The change from the inactive to the active region is not exactly equal in the three computations, but the obtained control is close after this region.\\

For the antiangiogenic therapy $\ControlS$, the behavior is also stable. A small difference appears near the initial and final times, but the inactive zone in the middle of the interval appears in the three runs. Also, the final activation of $\ControlS$ is recovered for all the initial controls. Therefore, the role of the antiangiogenic therapy in the optimal strategy is preserved when the starting control is changed.\\

Table \ref{tab:exp2} shows, for each initial control, the approximation of the objective functional $J$, the stopping iteration of the PINN algorithm, and the corresponding values obtained with the classical solver. The final values of $J$ are close in the three cases, with a difference of less than $0.5\%$, which agrees with the similar structure observed in the optimal controls. It can be observed that the classical reference also changes slightly with the initial control. This is an expected behavior, given that the problem is not convex and the algorithm gives local optimal controls.
\begin{table}[H]
\centering
\small
\begin{tabular}{l|c|c|c|c}
\hline
\textbf{Initial control} & \textbf{$J$ (HCS, 4 subintervals)} & \textbf{Iter Stop} & \textbf{$J$ (classic)} & \textbf{Iter Stop (classic)} \\ \hline
1 (constant)                  & $15.450080$ & $79$  & $15.350898$ & $118$ \\
2 (concentrated at the start) & $15.407012$ & $153$ & $15.364730$ & $162$ \\
3 (sinusoidal)                & $15.385723$ & $98$  & $15.349983$ & $116$ \\ \hline
\end{tabular}
\caption{Objective functional and stopping iteration for the three initial controls.}
\label{tab:exp2}
\end{table}

\section{Conclusions}

This article proposes and analyzes an ODE model to describe glioblastoma dynamics, considering the interaction between tumor density and oxygen concentration. The model incorporates two therapeutic interventions: chemotherapy and antiangiogenic therapy, allowing the representation of the effect of both treatments on tumor progression. Based on this model, an optimal control problem was formulated, with the aim of reducing the tumor density, as well as regulate oxygen concentration around a desired value, and penalize the excessive use of therapies. The well-posedness of the state system was proved, including the positivity and boundedness of the solutions. With respect to the optimal control problem, the existence of global optimal solution was obtained. Furthermore, by using Pontryagin Minimum Principle, the adjoint system and the necessary conditions for optimality were derived.\\

For the numerical approximation, the Physics-Informed Neural Networks methodology was adapted to the state-adjoint system associated with the control problem. Specifically, time-sequential formulations were considered, both with soft and hard- constraints. In addition, an error estimate for these approximations was deduced. This approach is particularly relevant, since the quality of the state approximation directly influences the approximation of the adjoint variables and, therefore, the updating of the controls. Numerical simulations show that the proposed methodology accurately reproduces the qualitative behavior of optimal controls and state variables, compared to a classical solver. Specifically, the sequential formulation with strong constraints yielded the most accurate results. Furthermore, the formulation with soft constraints was less computationally expensive, although less accurate.
\\ 

{\bf Credit  authorship contribution statement}\\
J.J. Forero-Hern\'andez: Writing-review \& editing, Writing-original draft, Visualization, Validation, Software, Methodology, Investigation, Formal analysis, Conceptualization. E.J. Villamizar-Roa: Writing-review \& editing, Writing-original draft, Visualization,
Validation, Methodology, Investigation, Formal analysis, Conceptualization.\\

{\bf Declaration of competing interest.}\\
The authors declare that they have no known competing financial interests or personal relationships that could have appeared
to influence the work reported in this paper.\\

{\bf Data availability.}\\
The code will be available in a public repository once the internal review and release process is completed.\\

{\bf Acknowledgments.}\\
The first author was supported by the Vicerrector\'ia Acad\'emica of the Universidad Industrial de Santander. The second author was supported by the Vicerrector\'ia de Investigaci\'on y Extenci\'on of the Universidad Industrial de Santander.  The authors would like to thank Professor Francisco Guill\'en-Gonz\'alez for providing useful comments on the manuscript.

%%%%%%%%%%%%%%%%%%%%%%%%%%%%%%%%%%%%%%%%%%%%%%%%%%%%%%%%%%%%%
%%%%%%%%%%%%%%%%%%%%%%%%%%%%%%%%%%%%%%%%%%%%%%%%%%%%%%%%%%%%%
% BIBLIOGRAFIA
%%%%%%%%%%%%%%%%%%%%%%%%%%%%%%%%%%%%%%%%%%%%%%%%%%%%%%%%%%%%%
%%%%%%%%%%%%%%%%%%%%%%%%%%%%%%%%%%%%%%%%%%%%%%%%%%%%%%%%%%%%%

\end{document}

%% file: graficas/funcional_SC_PINNs_exp1.tex
\input{graficas/estilos_graficas}
\GraphSetupStyles{\linewidth}{0.48\linewidth}{fourStyle}{redTone}{oneStyle}{greenTone}
\begin{figure}[H]
\centering
% ================================================================
% ====================== Gráfica A ================================
% ================================================================
\begin{subfigure}[t]{0.48\linewidth}
\centering
\begin{tikzpicture}[
  spy using outlines={rectangle, % forma del recuadro en la zona seleccionada
    lens={
      scale=1.8, % zoom
    },
    width=3.2cm,
    height=1.05cm, % Tamaño de la ventana
    spy connection path={
      \draw[zoomLink]
        (tikzspyonnode.north east) -- ([xshift=-2.9cm]tikzspyinnode.south east);
      \draw[zoomLink]
        (tikzspyonnode.south east) -- (tikzspyinnode.south east);
       \draw[zoomLink]
        (tikzspyonnode.north west) -- (tikzspyinnode.north west);
       \draw[zoomLink]
        (tikzspyonnode.south west) -- (tikzspyinnode.south west);
    },
    every spy on node/.append style={zoomSourceFrame},
    every spy in node/.append style={zoomWindow}
  }
]
\begin{axis}[
  axisStyle,
  legend style={legendHidden},
  unbounded coords=jump,
  xlabel={t},
  ylabel={J},
  xmin=-4.9500000001308999,
  xmax=120.9500000001309,
  ymin=14,
  ymax=41
]
  \addplot+[refStyle] coordinates { (1,40.303177316357882) (2,21.908500508333738) (3,17.979008710736093) (4,17.243110524807197) (5,17.051361535742284) (6,16.955519619210964) (7,16.892608302197516) (8,16.848603941389189) (9,16.813477982255829) (10,16.78347050827378) (11,16.756633071694097) (12,16.732724970608992) (13,16.710744697519196) (14,16.690047904493383) (15,16.670521881992894) (16,16.65191464062352) (17,16.63437464909055) (18,16.617610631472157) (19,16.600879743053493) (20,16.585179038302421) (21,16.569755071768245) (22,16.554945489036726) (23,16.540235214647691) (24,16.525884549477912) (25,16.512483344508126) (26,16.498177508404165) (27,16.485043203981782) (28,16.4718711583599) (29,16.459325914119553) (30,16.446442363853141) (31,16.433620202014421) (32,16.421544760185288) (33,16.410336067838479) (34,16.398166004215522) (35,16.385778727532017) (36,16.374159207506832) (37,16.363346266534037) (38,16.352524331698504) (39,16.341714365589773) (40,16.330466686441) (41,16.319771732513143) (42,16.310483355272439) (43,16.301723300195651) (44,16.293460705029631) (45,16.285659868085048) (46,16.278285873169523) (47,16.115138540337334) (48,16.015058480470522) (49,15.949935270320282) (50,15.902065112125507) (51,15.865647747186426) (52,15.838234121514356) (53,15.81544279214528) (54,15.796143267369297) (55,15.779418119076512) (56,15.765988933467664) (57,15.75417336881809) (58,15.743406091396992) (59,15.733597800183395) (60,15.724564470459027) (61,15.716312738335386) (62,15.708685784667185) (63,15.70245891942135) (64,15.696472004183679) (65,15.690765746662469) (66,15.685285087861859) (67,15.680095305430791) (68,15.67505956540904) (69,15.670423822646207) (70,15.665681626386926) (71,15.66134867706965) (72,15.657852299757842) (73,15.654278278379705) (74,15.650980489140675) (75,15.647878421426295) (76,15.644677107366139) (77,15.641424933463519) (78,15.63844540555904) (79,15.635627864376486) (80,15.632795358043726) (81,15.629822034811969) (82,15.556936053226114) (83,15.516587550875942) (84,15.489912440122042) (85,15.470518879295451) (86,15.455647479609592) (87,15.443770852819387) (88,15.433992163401289) (89,15.425775839594953) (90,15.418721643944242) (91,15.412573164752258) (92,15.407117956756847) (93,15.402429432247475) (94,15.398511834137077) (95,15.394852854236003) (96,15.391443362057331) (97,15.388277585502525) (98,15.385288409854645) (99,15.382504178825073) (100,15.379928853329492) (101,15.377389805730395) (102,15.375094096071072) (103,15.372788876946485) (104,15.370706533964666) (105,15.368607702183404) (106,15.366690055274454) (107,15.364852732665996) (108,15.363040211433855) (109,15.361380341642816) (110,15.359715841349233) (111,15.358088126508816) (112,15.35658317508179) (113,15.355108824433374) (114,15.353588060108882) (115,15.352181335184941) (116,15.350897882278407) (117,15.350897882278407) (118,15.350897882278407) (119,15.350897882278407) };
  \addplot+[fourStyle] coordinates { (1,40.303455044591225) (2,21.897225765599575) (3,18.141479250945171) (4,17.312427490133178) (5,17.099465009975464) (6,16.996236751656696) (7,16.622674081742385) (8,16.460261444854375) (9,16.37069767591732) (10,16.317066811578361) (11,16.28050177694363) (12,16.251690030967481) (13,16.230377095093218) (14,16.20984794709338) (15,16.18997860155044) (16,16.175322592565788) (17,16.158155453504445) (18,16.145112928484565) (19,16.130466090094618) (20,16.119424039764237) (21,16.10505712404721) (22,15.87871362265189) (23,15.87871362265189) (24,15.87871362265189) (25,15.87871362265189) (26,15.87871362265189) (27,15.87871362265189) (28,15.87871362265189) (29,15.87871362265189) (30,15.87871362265189) (31,15.87871362265189) (32,15.87871362265189) (33,15.87871362265189) (34,15.87871362265189) (35,15.87871362265189) (36,15.87871362265189) (37,15.87871362265189) (38,15.87871362265189) (39,15.87871362265189) (40,15.87871362265189) (41,15.87871362265189) (42,15.87871362265189) (43,15.87871362265189) (44,15.87871362265189) (45,15.87871362265189) (46,15.87871362265189) (47,15.87871362265189) (48,15.87871362265189) (49,15.87871362265189) (50,15.87871362265189) (51,15.87871362265189) (52,15.87871362265189) (53,15.87871362265189) (54,15.87871362265189) (55,15.87871362265189) (56,15.87871362265189) (57,15.87871362265189) (58,15.87871362265189) (59,15.87871362265189) (60,15.87871362265189) (61,15.87871362265189) (62,15.87871362265189) (63,15.87871362265189) (64,15.87871362265189) (65,15.87871362265189) (66,15.87871362265189) (67,15.87871362265189) (68,15.87871362265189) (69,15.87871362265189) (70,15.87871362265189) (71,15.87871362265189) (72,15.87871362265189) (73,15.87871362265189) (74,15.87871362265189) (75,15.87871362265189) (76,15.87871362265189) (77,15.87871362265189) (78,15.87871362265189) (79,15.87871362265189) (80,15.87871362265189) (81,15.87871362265189) (82,15.87871362265189) (83,15.87871362265189) (84,15.87871362265189) (85,15.87871362265189) (86,15.87871362265189) (87,15.87871362265189) (88,15.87871362265189) (89,15.87871362265189) (90,15.87871362265189) (91,15.87871362265189) (92,15.87871362265189) (93,15.87871362265189) (94,15.87871362265189) (95,15.87871362265189) (96,15.87871362265189) (97,15.87871362265189) (98,15.87871362265189) (99,15.87871362265189) (100,15.87871362265189) (101,15.87871362265189) (102,15.87871362265189) (103,15.87871362265189) (104,15.87871362265189) (105,15.87871362265189) (106,15.87871362265189) (107,15.87871362265189) (108,15.87871362265189) (109,15.87871362265189) (110,15.87871362265189) (111,15.87871362265189) (112,15.87871362265189) (113,15.87871362265189) (114,15.87871362265189) (115,15.87871362265189) (116,15.87871362265189) (117,15.87871362265189) (118,15.87871362265189) (119,15.87871362265189) };
  \addplot+[oneStyle] coordinates { (1,40.303370845386901) (2,21.873607208438099) (3,17.974097562998896) (4,17.248473887553125) (5,17.059060847468178) (6,16.958947634007824) (7,16.893025020854797) (8,16.856281442261086) (9,16.812274350152325) (10,16.783838513885343) (11,16.761239866842352) (12,16.733739771468301) (13,16.708734557489326) (14,16.66866837942267) (15,16.653037955632211) (16,16.420471649054818) (17,16.31103639998544) (18,16.235491683620992) (19,16.189917228771108) (20,16.141784210922609) (21,16.104342631334013) (22,16.07892102069118) (23,16.049024407202261) (24,16.03585041479047) (25,16.023162915806779) (26,16.002404817613559) (27,15.998481970889618) (28,15.980480116952608) (29,15.970749276964987) (30,15.878972907989501) (31,15.830194161005416) (32,15.793498746422648) (33,15.771256575956892) (34,15.755198031606671) (35,15.735528005607893) (36,15.726267779933293) (37,15.712145452676983) (38,15.705925594333694) (39,15.695958144818935) (40,15.692492079030581) (41,15.686707180082063) (42,15.676039710197927) (43,15.671170016123982) (44,15.665024652680039) (45,15.656117983722121) (46,15.656117983722121) (47,15.653558780633221) (48,15.653558780633221) (49,15.653558780633221) (50,15.64497528742163) (51,15.6423002325167) (52,15.638213949618368) (53,15.633398593573519) (54,15.62367272336185) (55,15.62367272336185) (56,15.62367272336185) (57,15.62367272336185) (58,15.62367272336185) (59,15.62367272336185) (60,15.62367272336185) (61,15.62367272336185) (62,15.62367272336185) (63,15.62367272336185) (64,15.62367272336185) (65,15.62367272336185) (66,15.62367272336185) (67,15.62367272336185) (68,15.62367272336185) (69,15.62367272336185) (70,15.62367272336185) (71,15.62367272336185) (72,15.62367272336185) (73,15.62367272336185) (74,15.62367272336185) (75,15.62367272336185) (76,15.62367272336185) (77,15.62367272336185) (78,15.62367272336185) (79,15.62367272336185) (80,15.62367272336185) (81,15.62367272336185) (82,15.62367272336185) (83,15.62367272336185) (84,15.62367272336185) (85,15.62367272336185) (86,15.62367272336185) (87,15.62367272336185) (88,15.62367272336185) (89,15.62367272336185) (90,15.62367272336185) (91,15.62367272336185) (92,15.62367272336185) (93,15.62367272336185) (94,15.62367272336185) (95,15.62367272336185) (96,15.62367272336185) (97,15.62367272336185) (98,15.62367272336185) (99,15.62367272336185) (100,15.62367272336185) (101,15.62367272336185) (102,15.62367272336185) (103,15.62367272336185) (104,15.62367272336185) (105,15.62367272336185) (106,15.62367272336185) (107,15.62367272336185) (108,15.62367272336185) (109,15.62367272336185) (110,15.62367272336185) (111,15.62367272336185) (112,15.62367272336185) (113,15.62367272336185) (114,15.62367272336185) (115,15.62367272336185) (116,15.62367272336185) (117,15.62367272336185) (118,15.62367272336185) (119,15.62367272336185) };
  
  \coordinate (zoomPointC) at (axis cs:33,17.7); % Punto del grafico a ser ampliado (x=t, y=valor de C).
  \coordinate (zoomViewC) at (axis cs:73,28); % Ubicación de la ventana del zoom
\end{axis}
\spy on (zoomPointC) in node at (zoomViewC);
\end{tikzpicture}
\caption{Objective functional $J$.}
\label{fig:sc_functional_j}
\end{subfigure}\hfill%
% ================================================================
% ====================== Gráfica C ================================
% ================================================================
\begin{subfigure}[t]{0.48\linewidth}
\centering
\begin{tikzpicture}
\begin{axis}[
  axisStyle,
  legend style={legendHidden},
  unbounded coords=jump,
  xlabel={t},
  ylabel={$|\Delta J|$},
  xmin=-4.9500000001308999,
  xmax=125.9500000001309,
  ymin=-0.039729476714892464,
  ymax=0.83798261107236238
]
  \addplot+[fourStyle] coordinates { (1,0.00027772823334260011) (2,0.011274742734162402) (3,0.16247054020907825) (4,0.069316965325981528) (5,0.048103474233180066) (6,0.040717132445731608) (7,0.26993422045513071) (8,0.38834249653481479) (9,0.44278030633850918) (10,0.46640369669541926) (11,0.47613129475046634) (12,0.48103493964151056) (13,0.48036760242597865) (14,0.48019995740000354) (15,0.48054328044245409) (16,0.47659204805773214) (17,0.47621919558610415) (18,0.47249770298759231) (19,0.47041365295887516) (20,0.46575499853818414) (21,0.46469794772103512) (22,0.67623186638483546) (23,0.66152159199580041) (24,0.64717092682602129) (25,0.63376972185623615) (26,0.61946388575227473) (27,0.60632958132989145) (28,0.59315753570800922) (29,0.5806122914676628) (30,0.56772874120125039) (31,0.55490657936253029) (32,0.54283113753339762) (33,0.53162244518658852) (34,0.5194523815636316) (35,0.50706510488012668) (36,0.4954455848549415) (37,0.48463264388214711) (38,0.47381070904661371) (39,0.46300074293788285) (40,0.45175306378910918) (41,0.44105810986125249) (42,0.4317697326205483) (43,0.42300967754376018) (44,0.41474708237774038) (45,0.40694624543315783) (46,0.39957225051763245) (47,0.2364249176854436) (48,0.13634485781863148) (49,0.071221647668391341) (50,0.023351489473617093) (51,0.013065875465464671) (52,0.040479501137534513) (53,0.063270830506610309) (54,0.082570355282593155) (55,0.099295503575378419) (56,0.11272468918422618) (57,0.12454025383379985) (58,0.13530753125489881) (59,0.14511582246849564) (60,0.15414915219286307) (61,0.16240088431650435) (62,0.17002783798470489) (63,0.1762547032305406) (64,0.18224161846821119) (65,0.18794787598942087) (66,0.19342853479003175) (67,0.19861831722109891) (68,0.20365405724285068) (69,0.20828980000568365) (70,0.21303199626496472) (71,0.21736494558223995) (72,0.22086132289404858) (73,0.22443534427218559) (74,0.22773313351121516) (75,0.2308352012255952) (76,0.23403651528575153) (77,0.23728868918837165) (78,0.24026821709285073) (79,0.24308575827540402) (80,0.24591826460816435) (81,0.24889158783992116) (82,0.32177756942577673) (83,0.36212607177594869) (84,0.38880118252984808) (85,0.40819474335643946) (86,0.42306614304229839) (87,0.43494276983250302) (88,0.44472145925060147) (89,0.45293778305693699) (90,0.45999197870764874) (91,0.46614045789963221) (92,0.47159566589504287) (93,0.4762841904044155) (94,0.48020178851481354) (95,0.48386076841588732) (96,0.4872702605945598) (97,0.49043603714936523) (98,0.49342521279724494) (99,0.49620944382681742) (100,0.49878476932239835) (101,0.50132381692149508) (102,0.50361952658081854) (103,0.50592474570540524) (104,0.50800708868722388) (105,0.51010592046848657) (106,0.51202356737743671) (107,0.51386088998589408) (108,0.51567341121803523) (109,0.5173332810090745) (110,0.51899778130265695) (111,0.52062549614307407) (112,0.52213044757010074) (113,0.52360479821851591) (114,0.52512556254300868) (115,0.52653228746694936) (116,0.52781574037348378) (117,0.52781574037348378) (118,0.52781574037348378) (119,0.52781574037348378) };
  \addplot+[oneStyle] coordinates { (1,0.00019352902901914604) (2,0.034893299895639274) (3,0.0049111477371965861) (4,0.0053633627459284128) (5,0.007699311725893665) (6,0.0034280147968601682) (7,0.00041671865728076796) (8,0.0076775008718961146) (9,0.0012036321035040487) (10,0.00036800561156269396) (11,0.0046067951482555713) (12,0.0010148008593091618) (13,0.0020101400298706551) (14,0.021379525070713612) (15,0.017483926360682744) (16,0.23144299156870218) (17,0.32333824910510955) (18,0.38211894785116485) (19,0.41096251428238517) (20,0.44339482737981228) (21,0.46541244043423191) (22,0.47602446834554613) (23,0.49121080744543022) (24,0.49003413468744128) (25,0.48932042870134751) (26,0.49577269079060571) (27,0.48656123309216426) (28,0.49139104140729195) (29,0.48857663715456567) (30,0.56746945586363928) (31,0.60342604100900488) (32,0.62804601376264024) (33,0.63907949188158675) (34,0.64296797260885086) (35,0.65025072192412381) (36,0.64789142757353879) (37,0.65120081385705397) (38,0.64659873736481011) (39,0.6457562207708385) (40,0.63797460741041867) (41,0.63306455243107962) (42,0.63444364507451212) (43,0.63055328407166833) (44,0.62843605234959199) (45,0.62954188436292746) (46,0.62216788944740209) (47,0.46157975970411336) (48,0.36149969983730124) (49,0.2963764896870611) (50,0.25708982470387731) (51,0.22334751466972591) (52,0.20002017189598753) (53,0.18204419857176113) (54,0.17247054400744766) (55,0.15574539571466239) (56,0.14231621010581463) (57,0.13050064545624096) (58,0.11973336803514201) (59,0.10992507682154518) (60,0.10089174709717774) (61,0.092640014973536466) (62,0.085013061305335924) (63,0.078786196059500213) (64,0.072799280821829626) (65,0.067093023300619947) (66,0.061612364500009065) (67,0.056422582068941907) (68,0.051386842047190129) (69,0.046751099284357167) (70,0.042008903025076094) (71,0.037675953707800858) (72,0.034179576395992228) (73,0.030605555017855224) (74,0.027307765778825654) (75,0.024205698064445613) (76,0.021004384004289278) (77,0.017752210101669164) (78,0.014772682197190079) (79,0.011955141014636794) (80,0.0091226346818764625) (81,0.0061493114501196544) (82,0.06673667013573592) (83,0.10708517248590788) (84,0.13376028323980726) (85,0.15315384406639865) (86,0.16802524375225758) (87,0.17990187054246221) (88,0.18968055996056066) (89,0.19789688376689618) (90,0.20495107941760793) (91,0.2110995586095914) (92,0.21655476660500206) (93,0.22124329111437468) (94,0.22516088922477273) (95,0.22881986912584651) (96,0.23222936130451899) (97,0.23539513785932442) (98,0.23838431350720413) (99,0.24116854453677661) (100,0.24374387003235753) (101,0.24628291763145427) (102,0.24857862729077773) (103,0.25088384641536443) (104,0.25296618939718307) (105,0.25506502117844576) (106,0.25698266808739589) (107,0.25881999069585326) (108,0.26063251192799441) (109,0.26229238171903368) (110,0.26395688201261613) (111,0.26558459685303326) (112,0.26708954828005993) (113,0.2685638989284751) (114,0.27008466325296787) (115,0.27149138817690854) (116,0.27277484108344296) (117,0.27277484108344296) (118,0.27277484108344296) (119,0.27277484108344296) };
\end{axis}

\end{tikzpicture}
\caption{Error of $J$.}
\label{fig:sc_functional_error_j}
\end{subfigure}

\vspace{0.25em}
\centering
% Leyenda compartida (aplica a los 4 paneles de la figura principal).
\begin{tikzpicture}
  \begin{axis}[
    hide axis,
    xmin=0, xmax=1, ymin=0, ymax=1,
    legend columns=3,
    legend style={
      legendBox,
      at={(0.5,0.5)},
      anchor=center,
      font=\small
    }
  ]
    \addlegendimage{refStyle}
    \addlegendentry{Classical reference}
    \addlegendimage{fourStyle}
    \addlegendentry{Sequential}
    \addlegendimage{oneStyle}
    \addlegendentry{Non-sequential}
  \end{axis}
\end{tikzpicture}
\vspace{0.18em}
\caption{Comparison of the objective functional $J$ and its error with Soft-Constrained PINNs. The classical method vs the sequential scheme with 4 subintervals vs the non-sequential scheme with 1 subinterval.}
\label{fig:sc_functional_panel}
\end{figure}

%% file: graficas/funcional_HC_PINNs_exp1.tex
\input{graficas/estilos_graficas}
\GraphSetupStyles{\linewidth}{0.48\linewidth}{fourStyle}{redTone}{oneStyle}{greenTone}
\begin{figure}[H]
\centering
% ================================================================
% ====================== Gráfica A ================================
% ================================================================
\begin{subfigure}[t]{0.48\linewidth}
\centering
\begin{tikzpicture}[
  spy using outlines={rectangle, % forma del recuadro en la zona seleccionada
    lens={
      scale=1.8, % zoom
    },
    width=3.2cm,
    height=1.05cm, % Tamaño de la ventana
    spy connection path={
      \draw[zoomLink]
        (tikzspyonnode.north east) -- ([xshift=-2.9cm]tikzspyinnode.south east);
      \draw[zoomLink]
        (tikzspyonnode.south east) -- (tikzspyinnode.south east);
       \draw[zoomLink]
        (tikzspyonnode.north west) -- (tikzspyinnode.north west);
       \draw[zoomLink]
        (tikzspyonnode.south west) -- (tikzspyinnode.south west);
    },
    every spy on node/.append style={zoomSourceFrame},
    every spy in node/.append style={zoomWindow}
  }
]
\begin{axis}[
  axisStyle,
  legend style={legendHidden},
  unbounded coords=jump,
  xlabel={t},
  ylabel={J},
  xmin=-4.9500000001308999,
  xmax=125.9500000001309,
  ymin=14,
  ymax=41.248544093728455
]
  \addplot+[refStyle] coordinates { (1,40.303177316357882) (2,21.908500508333738) (3,17.979008710736093) (4,17.243110524807197) (5,17.051361535742284) (6,16.955519619210964) (7,16.892608302197516) (8,16.848603941389189) (9,16.813477982255829) (10,16.78347050827378) (11,16.756633071694097) (12,16.732724970608992) (13,16.710744697519196) (14,16.690047904493383) (15,16.670521881992894) (16,16.65191464062352) (17,16.63437464909055) (18,16.617610631472157) (19,16.600879743053493) (20,16.585179038302421) (21,16.569755071768245) (22,16.554945489036726) (23,16.540235214647691) (24,16.525884549477912) (25,16.512483344508126) (26,16.498177508404165) (27,16.485043203981782) (28,16.4718711583599) (29,16.459325914119553) (30,16.446442363853141) (31,16.433620202014421) (32,16.421544760185288) (33,16.410336067838479) (34,16.398166004215522) (35,16.385778727532017) (36,16.374159207506832) (37,16.363346266534037) (38,16.352524331698504) (39,16.341714365589773) (40,16.330466686441) (41,16.319771732513143) (42,16.310483355272439) (43,16.301723300195651) (44,16.293460705029631) (45,16.285659868085048) (46,16.278285873169523) (47,16.115138540337334) (48,16.015058480470522) (49,15.949935270320282) (50,15.902065112125507) (51,15.865647747186426) (52,15.838234121514356) (53,15.81544279214528) (54,15.796143267369297) (55,15.779418119076512) (56,15.765988933467664) (57,15.75417336881809) (58,15.743406091396992) (59,15.733597800183395) (60,15.724564470459027) (61,15.716312738335386) (62,15.708685784667185) (63,15.70245891942135) (64,15.696472004183679) (65,15.690765746662469) (66,15.685285087861859) (67,15.680095305430791) (68,15.67505956540904) (69,15.670423822646207) (70,15.665681626386926) (71,15.66134867706965) (72,15.657852299757842) (73,15.654278278379705) (74,15.650980489140675) (75,15.647878421426295) (76,15.644677107366139) (77,15.641424933463519) (78,15.63844540555904) (79,15.635627864376486) (80,15.632795358043726) (81,15.629822034811969) (82,15.556936053226114) (83,15.516587550875942) (84,15.489912440122042) (85,15.470518879295451) (86,15.455647479609592) (87,15.443770852819387) (88,15.433992163401289) (89,15.425775839594953) (90,15.418721643944242) (91,15.412573164752258) (92,15.407117956756847) (93,15.402429432247475) (94,15.398511834137077) (95,15.394852854236003) (96,15.391443362057331) (97,15.388277585502525) (98,15.385288409854645) (99,15.382504178825073) (100,15.379928853329492) (101,15.377389805730395) (102,15.375094096071072) (103,15.372788876946485) (104,15.370706533964666) (105,15.368607702183404) (106,15.366690055274454) (107,15.364852732665996) (108,15.363040211433855) (109,15.361380341642816) (110,15.359715841349233) (111,15.358088126508816) (112,15.35658317508179) (113,15.355108824433374) (114,15.353588060108882) (115,15.352181335184941) (116,15.350897882278407) (117,15.350897882278407) (118,15.350897882278407) (119,15.350897882278407) };
  \addplot+[fourStyle] coordinates { (1,40.302785991948141) (2,21.900799496845668) (3,17.975065116217795) (4,17.237914972635782) (5,17.049018415163019) (6,16.951562641292451) (7,16.890585120046673) (8,16.846490126037523) (9,16.811293412817143) (10,16.782156102722926) (11,16.7540297929236) (12,16.731840453894318) (13,16.708579063516101) (14,16.688173905680806) (15,16.668027321420031) (16,16.649268708425236) (17,16.631691663512914) (18,16.615873859132538) (19,16.599317690202486) (20,16.583637534392729) (21,16.568037712367005) (22,16.552474108106473) (23,16.538174516083469) (24,16.52596162903734) (25,16.508310222940818) (26,16.49807175256538) (27,16.482167587252839) (28,16.470104534576333) (29,16.455604967132345) (30,16.443283076670269) (31,16.432057155391295) (32,16.253559922133284) (33,16.142720235132586) (34,16.072802882132365) (35,16.018594753717135) (36,15.981149555245651) (37,15.950460738228589) (38,15.927663393339245) (39,15.904563680182486) (40,15.885617362423519) (41,15.868325188887928) (42,15.853092964306285) (43,15.840558992881588) (44,15.828431679516358) (45,15.817007036216411) (46,15.803371812356312) (47,15.793925256850867) (48,15.786946641254284) (49,15.702840218571151) (50,15.655136404337492) (51,15.623574476533234) (52,15.600351488397907) (53,15.582788100877183) (54,15.568050714243084) (55,15.554991646352152) (56,15.545041996076552) (57,15.535942528151207) (58,15.525893193983682) (59,15.520203416092734) (60,15.513500477275018) (61,15.509188444156857) (62,15.502620245968943) (63,15.497296859856441) (64,15.492454415240907) (65,15.487599712781098) (66,15.480308414052395) (67,15.480308414052395) (68,15.480308414052395) (69,15.478956995992114) (70,15.474753884858899) (71,15.469058683964052) (72,15.469058683964052) (73,15.467524317567454) (74,15.465301451899494) (75,15.461379380220356) (76,15.457062525844147) (77,15.450079842716056) (78,15.450079842716056) (79,15.450079842716056) (80,15.450079842716056) (81,15.450079842716056) (82,15.450079842716056) (83,15.450079842716056) (84,15.450079842716056) (85,15.450079842716056) (86,15.450079842716056) (87,15.450079842716056) (88,15.450079842716056) (89,15.450079842716056) (90,15.450079842716056) (91,15.450079842716056) (92,15.450079842716056) (93,15.450079842716056) (94,15.450079842716056) (95,15.450079842716056) (96,15.450079842716056) (97,15.450079842716056) (98,15.450079842716056) (99,15.450079842716056) (100,15.450079842716056) (101,15.450079842716056) (102,15.450079842716056) (103,15.450079842716056) (104,15.450079842716056) (105,15.450079842716056) (106,15.450079842716056) (107,15.450079842716056) (108,15.450079842716056) (109,15.450079842716056) (110,15.450079842716056) (111,15.450079842716056) (112,15.450079842716056) (113,15.450079842716056) (114,15.450079842716056) (115,15.450079842716056) (116,15.450079842716056) (117,15.450079842716056) (118,15.450079842716056) (119,15.450079842716056) };
  \addplot+[oneStyle] coordinates { (1,40.303288372439177) (2,21.894596980433942) (3,17.976061559970923) (4,17.249119616347468) (5,17.051680863181691) (6,16.958428083385492) (7,16.858729610380465) (8,16.836221819084653) (9,16.802568846610306) (10,16.780501299215089) (11,16.764055737387118) (12,16.733954133020603) (13,16.709993298473034) (14,16.680520640694894) (15,16.670733022511616) (16,16.654530235012984) (17,16.632822330640138) (18,16.412248940600477) (19,16.288020041955679) (20,16.210901979716191) (21,16.155861524238947) (22,16.11675534785941) (23,16.084001607467375) (24,16.067688418773255) (25,16.035012142606931) (26,16.020173156071156) (27,16.000500741835641) (28,15.987957204726349) (29,15.971224776260163) (30,15.955342471765862) (31,15.948242941856238) (32,15.937326394598992) (33,15.92012192884442) (34,15.914709992818704) (35,15.908245017149042) (36,15.899662343796601) (37,15.88839046576221) (38,15.881068784753863) (39,15.872876300214106) (40,15.865809280580539) (41,15.857436223504624) (42,15.85228021578704) (43,15.843073046325935) (44,15.835975636613318) (45,15.824434291884851) (46,15.752732664362791) (47,15.709737996586989) (48,15.681075316490684) (49,15.661726388102149) (50,15.642885286455249) (51,15.621526357066241) (52,15.614544156472025) (53,15.604257981848328) (54,15.600206759972542) (55,15.591473328427115) (56,15.582844734317046) (57,15.578442367410746) (58,15.561244276247567) (59,15.561244276247567) (60,15.561244276247567) (61,15.561244276247567) (62,15.561244276247567) (63,15.561244276247567) (64,15.561244276247567) (65,15.561244276247567) (66,15.561244276247567) (67,15.561244276247567) (68,15.561244276247567) (69,15.561244276247567) (70,15.561244276247567) (71,15.561244276247567) (72,15.561244276247567) (73,15.561244276247567) (74,15.561244276247567) (75,15.561244276247567) (76,15.561244276247567) (77,15.561244276247567) (78,15.561244276247567) (79,15.561244276247567) (80,15.561244276247567) (81,15.561244276247567) (82,15.561244276247567) (83,15.561244276247567) (84,15.561244276247567) (85,15.561244276247567) (86,15.561244276247567) (87,15.561244276247567) (88,15.561244276247567) (89,15.561244276247567) (90,15.561244276247567) (91,15.561244276247567) (92,15.561244276247567) (93,15.561244276247567) (94,15.561244276247567) (95,15.561244276247567) (96,15.561244276247567) (97,15.561244276247567) (98,15.561244276247567) (99,15.561244276247567) (100,15.561244276247567) (101,15.561244276247567) (102,15.561244276247567) (103,15.561244276247567) (104,15.561244276247567) (105,15.561244276247567) (106,15.561244276247567) (107,15.561244276247567) (108,15.561244276247567) (109,15.561244276247567) (110,15.561244276247567) (111,15.561244276247567) (112,15.561244276247567) (113,15.561244276247567) (114,15.561244276247567) (115,15.561244276247567) (116,15.561244276247567) (117,15.561244276247567) (118,15.561244276247567) (119,15.561244276247567) };
  
  \coordinate (zoomPointC) at (axis cs:33,17.7); % Punto del grafico a ser ampliado (x=t, y=valor de C).
  \coordinate (zoomViewC) at (axis cs:73,28); % Ubicación de la ventana del zoom
\end{axis}
\spy on (zoomPointC) in node at (zoomViewC);
\end{tikzpicture}
\caption{Objective functional $J$.}
\label{fig:hc_functional_j}
\end{subfigure}\hfill%
% ================================================================
% ====================== Gráfica C ================================
% ================================================================
\begin{subfigure}[t]{0.48\linewidth}
\centering
\begin{tikzpicture}
\begin{axis}[
  axisStyle,
  legend style={legendHidden},
  unbounded coords=jump,
  xlabel={t},
  ylabel={$|\Delta J|$},
  xmin=-4.9500000001308999,
  xmax=125.9500000001309,
  ymin=-0.039729476714892464,
  ymax=0.55
]
  \addplot+[fourStyle] coordinates { (1,0.00039132440974043448) (2,0.0077010114880700087) (3,0.0039435945182972887) (4,0.0051955521714148745) (5,0.0023431205792654453) (6,0.003956977918512905) (7,0.0020231821508431835) (8,0.0021138153516666591) (9,0.002184569438686168) (10,0.0013144055508540475) (11,0.0026032787704970417) (12,0.00088451671467382198) (13,0.0021656340030951071) (14,0.0018739988125773266) (15,0.0024945605728632358) (16,0.0026459321982841288) (17,0.0026829855776355771) (18,0.0017367723396191082) (19,0.0015620528510069676) (20,0.0015415039096922101) (21,0.0017173594012405147) (22,0.0024713809302525647) (23,0.0020606985642217523) (24,7.7079559428483435e-05) (25,0.0041731215673088684) (26,0.00010575583878491557) (27,0.0028756167289429868) (28,0.0017666237835669563) (29,0.0037209469872081513) (30,0.0031592871828713953) (31,0.0015630466231257856) (32,0.16798483805200348) (33,0.26761583270589284) (34,0.32536312208315721) (35,0.36718397381488188) (36,0.39300965226118123) (37,0.41288552830544845) (38,0.42486093835925942) (39,0.4371506854072873) (40,0.4448493240174809) (41,0.45144654362521486) (42,0.45739039096615386) (43,0.46116430731406233) (44,0.46502902551327274) (45,0.46865283186863671) (46,0.47491406081321053) (47,0.32121328348646649) (48,0.22811183921623801) (49,0.24709505174913104) (50,0.24692870778801534) (51,0.24207327065319184) (52,0.23788263311644897) (53,0.23265469126809712) (54,0.22809255312621346) (55,0.22442647272436034) (56,0.22094693739111193) (57,0.21823084066688381) (58,0.21751289741330915) (59,0.21339438409066069) (60,0.21106399318400904) (61,0.20712429417852896) (62,0.20606553869824218) (63,0.20516205956490907) (64,0.20401758894277222) (65,0.20316603388137189) (66,0.20497667380946361) (67,0.19978689137839645) (68,0.19475115135664467) (69,0.19146682665409287) (70,0.19092774152802683) (71,0.19228999310559836) (72,0.18879361579378973) (73,0.1867539608122506) (74,0.18567903724118118) (75,0.18649904120593952) (76,0.1876145815219914) (77,0.19134509074746298) (78,0.18836556284298389) (79,0.18554802166043061) (80,0.18271551532767027) (81,0.17974219209591347) (82,0.10685621051005789) (83,0.066507708159885937) (84,0.039832597405986547) (85,0.020439036579395164) (86,0.0055676368935362319) (87,0.0063089898966683933) (88,0.016087679314766845) (89,0.024304003121102369) (90,0.031358198771814116) (91,0.037506677963797586) (92,0.042961885959208246) (93,0.047650410468580873) (94,0.051568008578978919) (95,0.055226988480052697) (96,0.058636480658725176) (97,0.061802257213530609) (98,0.064791432861410314) (99,0.067575663890982796) (100,0.070150989386563722) (101,0.072690036985660456) (102,0.074985746644983919) (103,0.077290965769570619) (104,0.07937330875138926) (105,0.081472140532651949) (106,0.083389787441602081) (107,0.085227110050059451) (108,0.087039631282200602) (109,0.088699501073239873) (110,0.090364001366822322) (111,0.091991716207239449) (112,0.093496667634266117) (113,0.094971018282681285) (114,0.096491782607174059) (115,0.097898507531114731) (116,0.099181960437649153) (117,0.099181960437649153) (118,0.099181960437649153) (119,0.099181960437649153) };
  \addplot+[oneStyle] coordinates { (1,0.00011105608129469147) (2,0.013903527899795876) (3,0.0029471507651699369) (4,0.0060090915402710721) (5,0.00031932743940643604) (6,0.0029084641745278361) (7,0.033878691817051276) (8,0.012382122304536836) (9,0.010909135645523094) (10,0.0029692090586905806) (11,0.0074226656930207469) (12,0.0012291624116116395) (13,0.00075139904616250419) (14,0.0095272637984891162) (15,0.00021114051872217487) (16,0.0026155943894643485) (17,0.0015523184504111498) (18,0.20536169087167977) (19,0.31285970109781402) (20,0.37427705858623028) (21,0.41389354752929819) (22,0.4381901411773157) (23,0.45623360718031591) (24,0.45819613070465692) (25,0.47747120190119574) (26,0.47800435233300931) (27,0.48454246214614116) (28,0.48391395363355016) (29,0.48810113785939002) (30,0.49109989208727889) (31,0.48537726015818272) (32,0.48421836558629572) (33,0.49021413899405886) (34,0.48345601139681804) (35,0.47753371038297487) (36,0.4744968637102307) (37,0.47495580077182709) (38,0.47145554694464131) (39,0.46883806537566741) (40,0.46465740586046067) (41,0.46233550900851839) (42,0.45820313948539848) (43,0.45865025386971503) (44,0.45748506841631276) (45,0.46122557620019755) (46,0.52555320880673229) (47,0.40540054375034451) (48,0.33398316397983763) (49,0.28820888221813235) (50,0.25917982567025888) (51,0.24412139012018486) (52,0.22368996504233074) (53,0.21118481029695246) (54,0.19593650739675539) (55,0.18794479064939651) (56,0.18314419915061819) (57,0.17573100140734432) (58,0.1821618151494242) (59,0.17235352393582737) (60,0.16332019421145993) (61,0.15506846208781866) (62,0.14744150841961812) (63,0.14121464317378241) (64,0.13522772793611182) (65,0.12952147041490214) (66,0.12404081161429126) (67,0.1188510291832241) (68,0.11381528916147232) (69,0.10917954639863936) (70,0.10443735013935829) (71,0.10010440082208305) (72,0.096608023510274421) (73,0.093034002132137417) (74,0.089736212893107847) (75,0.086634145178727806) (76,0.083432831118571471) (77,0.080180657215951356) (78,0.077201129311472272) (79,0.074383588128918987) (80,0.071551081796158655) (81,0.068577758564401847) (82,0.0043082230214537276) (83,0.044656725371625683) (84,0.071331836125525072) (85,0.090725396952116455) (86,0.10559679663797539) (87,0.11747342342818001) (88,0.12725211284627846) (89,0.13546843665261399) (90,0.14252263230332574) (91,0.1486711114953092) (92,0.15412631949071987) (93,0.15881484400009249) (94,0.16273244211049054) (95,0.16639142201156432) (96,0.1698009141902368) (97,0.17296669074504223) (98,0.17595586639292193) (99,0.17874009742249442) (100,0.18131542291807534) (101,0.18385447051717208) (102,0.18615018017649554) (103,0.18845539930108224) (104,0.19053774228290088) (105,0.19263657406416357) (106,0.1945542209731137) (107,0.19639154358157107) (108,0.19820406481371222) (109,0.19986393460475149) (110,0.20152843489833394) (111,0.20315614973875107) (112,0.20466110116577774) (113,0.2061354518141929) (114,0.20765621613868568) (115,0.20906294106262635) (116,0.21034639396916077) (117,0.21034639396916077) (118,0.21034639396916077) (119,0.21034639396916077) };
\end{axis}

\end{tikzpicture}
\caption{Error of $J$.}
\label{fig:hc_functional_error_j}
\end{subfigure}

\vspace{0.25em}
\centering
% Leyenda compartida (aplica a los 4 paneles de la figura principal).
\begin{tikzpicture}
  \begin{axis}[
    hide axis,
    xmin=0, xmax=1, ymin=0, ymax=1,
    legend columns=3,
    legend style={
      legendBox,
      at={(0.5,0.5)},
      anchor=center,
      font=\small
    }
  ]
    \addlegendimage{refStyle}
    \addlegendentry{Classical reference}
    \addlegendimage{fourStyle}
    \addlegendentry{Sequential}
    \addlegendimage{oneStyle}
    \addlegendentry{Non-sequential}
  \end{axis}
\end{tikzpicture}
\vspace{0.18em}
\caption{Comparison of the objective functional $J$ and its error with Hard-Constrained PINNs. The classical method vs the sequential scheme with 4 subintervals vs and the non-sequential scheme with 1 subinterval.}
\label{fig:hc_functional_panel}
\end{figure}